\documentclass [12pt]{article}   

\usepackage{comment}

\usepackage{a4wide}    

\usepackage[utf8]{inputenc}   
\usepackage[T1]{fontenc}            
\usepackage[english, french, english]{babel}       
\usepackage{textcomp}                
\usepackage{graphicx}       
\usepackage{listing}        

\usepackage{amsthm}
\usepackage{amsmath}
\usepackage{amssymb}
\usepackage{amsfonts}
\usepackage{mathrsfs}
\usepackage{dsfont}
\usepackage{mathabx}
\usepackage{envmath}
\usepackage{cases}

\usepackage{epic, ecltree}
\usepackage{epsfig}
\usepackage{graphics}
\usepackage{color}

\usepackage{pgf,tikz,pgfplots} 
\usetikzlibrary{patterns}
\usepackage{mathrsfs}
\usetikzlibrary{arrows}

\makeatletter
\pgfdeclarepatternformonly[\LineSpace]{my north east lines}{\pgfqpoint{-1pt}{-1pt}}{\pgfpoint{\LineSpace+0.2pt}{\LineSpace+0.2pt}}{\pgfqpoint{\LineSpace}{\LineSpace}}%
{
  \pgfsetcolor{\tikz@pattern@color}
  \pgfsetlinewidth{0.4pt}
  \pgfpathmoveto{\pgfqpoint{0pt}{0pt}}
  \pgfpathlineto{\pgfqpoint{\LineSpace + 0.1pt}{\LineSpace + 0.1pt}}
  \pgfusepath{stroke}
}
\pgfdeclarepatternformonly[\LineSpace]{my north west lines}{\pgfqpoint{-1pt}{-1pt}}
{\pgfqpoint{\LineSpace}{\LineSpace}}{\pgfqpoint{\LineSpace}{\LineSpace}}%
{
  \pgfsetcolor{\tikz@pattern@color}
  \pgfsetlinewidth{0.4pt}
  \pgfpathmoveto{\pgfqpoint{0pt}{\LineSpace}}
  \pgfpathlineto{\pgfqpoint{\LineSpace + 0.1pt}{-0.1pt}}
  \pgfusepath{stroke}
}

\newdimen\LineSpace
\tikzset{
  line space/.code={\LineSpace=#1},
  line space=3pt
}
\makeatother

\usepackage{fancyhdr}     

\usepackage{lmodern}
\DeclareFontFamily{OMX}{lmex}{}
\DeclareFontShape{OMX}{lmex}{m}{n}{<->lmex10}{}

\usepackage[colorlinks=true, linkcolor=black, urlcolor=blue, filecolor=black, citecolor=black, menucolor=black]{hyperref}  

\usepackage{enumitem} 

\newcommand{\enstq}[2]{\left\lbrace#1\mathrel{}\middle|\mathrel{}#2\right\rbrace} 

\newcommand{\ensemblenombre}[1]{\mathbb{#1}} 
\newcommand{\N}{\ensemblenombre{N}}

\newcommand{\R}{\ensemblenombre{R}}
\newcommand{\C}{\ensemblenombre{C}}
\newcommand{\T}{\ensemblenombre{T}}
\newcommand{\D}{\ensemblenombre{D}}

\newcommand{\Hol}{\operatorname{Hol}}
\newcommand{\Ran}{\operatorname{Ran}}

\newcommand{\dist}{\operatorname{dist}}

\newcommand{\abs}[1]{\left\lvert#1\right\rvert} 
\newcommand{\prodscal}[2]{\left\langle#1,#2\right\rangle} 

\makeatletter
\newcommand{\bigint}{\@ifnextchar_\@bigintsub\@bigintnosub}
\def\@bigintsub_#1{\def\@int@subscript{#1}\@ifnextchar^\@bigintsubsup\@bigintsubnosup}
\def\@bigintsubsup^#1{\mathop{\text{\large$\int_{\text{\normalsize$\scriptstyle\@int@subscript$}}^{\text{\normalsize$\scriptstyle#1$}}$}}\nolimits}
\def\@bigintsubnosup{\mathop{\text{\large$\int_{\text{\normalsize$\scriptstyle\@int@subscript$}}$}}\nolimits}
\def\@bigintnosub{\@ifnextchar^\@bigintnosubsup\@bigintnosubnosup}
\def\@bigintnosubsup^#1{\mathop{\text{\large$\int^{\text{\normalsize$\scriptstyle#1$}}$}}\nolimits}
\def\@bigintnosubnosup{\mathop{\text{\large$\int$}}\nolimits}
\makeatother

\newtheorem{thm}{Theorem}[section]
\newtheorem{prop}[thm]{Proposition}

\newtheorem{cor}[thm]{Corollary }
\newtheorem{lem}[thm]{Lemma}

\theoremstyle{definition} 
\newtheorem{rmk}[thm]{Remark}
\newtheorem{ex}[thm]{Exemples}

\makeatletter
\newcommand{\subjclass}[2][2010]{%
  \let\@oldtitle\@title%
  \gdef\@title{\@oldtitle\footnotetext{#1 \emph{Mathematics subject classification.} #2}}%
}
\newcommand{\keywords}[1]{%
  \let\@@oldtitle\@title%
  \gdef\@title{\@@oldtitle\footnotetext{\emph{Key words and phrases.} #1.}}%
}
\makeatother

\newcounter{qu}

\newcounter{qq}
\newcounter{ex}
\title{Separation of singularities for the Bergman space and application to control theory}
\author{Andreas \bsc{Hartmann}\thanks{This research was partially supported by the project ANR-18-CE40-005 and by the Joint French-Russian Research Project PRC CNRS/RFBR 2017--2019.} \footnote{Université de Bordeaux, CNRS, Bordeaux INP, IMB, UMR 5251, 351 Cours de la Libération, F-33400, Talence, France, andreas.hartmann@math.u-bordeaux.fr}
\and Marcu-Antone \bsc{Orsoni}\footnotemark[1] \footnote{Université de Bordeaux, CNRS, Bordeaux INP, IMB, UMR 5251, 351 Cours de la Libération, F-33400, Talence, France, marcu-antone.orsoni@math.u-bordeaux.fr}}

\date{\today}


\usepackage{remreset}
\makeatletter \@removefromreset{figure}{subsection}\makeatother
\makeatletter \@removefromreset{figure}{section}\makeatother

\begin{document}

\maketitle

%
\abstract{In this paper, we solve a separation of singularities problem in the Bergman space. More precisely, we show that if $\mathcal{P} \subset \C$ is a convex polygon which is the intersection of  $n$ half planes,  
then the Bergman space on $\mathcal{P}$ decomposes into the sum of the Bergman spaces on these half planes. The result applies to the characterization of the reachable space of the one-dimensional heat equation on a finite interval with boundary controls. We prove that this space is a Bergman space of the square which has the given interval as a diagonal. This gives an affirmative answer to a conjecture raised in \cite{HKT}}.


\section{Introduction}
For $\Omega \subset \C$, an open set in the complex plane, we denote by $\mathrm{Hol}(\Omega)$ the space of holomorphic functions on $\Omega$. 
Given $\Omega_1$ and $\Omega_2$ two open subsets of $\C$ with non empty intersection, a natural question is to know whether every function $f \in \mathrm{Hol}(\Omega_1 \cap \Omega_2)$ can be written as a sum of two functions $f_1 \in  \mathrm{Hol}(\Omega_1)$ and $f_2 \in \mathrm{Hol}(\Omega_2)$, i.e.\ does the equality $\mathrm{Hol}(\Omega_1 \cap \Omega_2)=\mathrm{Hol}(\Omega_1) + \mathrm{Hol}(\Omega_2)$ hold ? This problem is known as the separation of singularities problem for holomorphic functions and has a quite long
history.
A simple example is given by $\Omega_2=\enstq{z \in \C}{\abs{z} < r_2}$ and $\Omega_1=\enstq{z \in \C}{\abs{z} > r_1}$ with $0 <r_1 <r_2$. Then $\Omega_1 \cap \Omega_2$ is a ring and the problem can be solved affirmatively using Laurent series. Poincaré \cite[V, Ch. 3, § 21]{Po} discussed the solution in the particular case when $\Omega_1=\C \setminus [-1,\, 1]$ and $\Omega_2 = \C \setminus \left((-\infty, \, -1] \cup [1, \, + \infty) \right)$, and Aronszajn \cite{Ar} gave a positive answer for any pair $(\Omega_1, \Omega_2)$ of open sets in $\C$. 

The separation of singularities problem is a special case of the First Cousin Problem 
which reduces the problem to solving a $\overline{\partial}$-equation and can be reformulated in sheaf cohomology terms (see \cite[Thm 1.4.5 and Thm 5.5.1]{Ho2}). The First Cousin problem has been solved a few years after Aronszajn, first by Oka \cite{Ok} on domains of holomorphy and then in the Cartan seminar \cite{Ca} on Stein manifolds. Today, it is well-known that the First Cousin Problem on $\C$ is equivalent to the Mittag-Leffler theorem (see \cite[pp. 11-14]{Ho2} and \cite[section 9.4]{AM}). 

We would also like to mention two other simple proofs of the separation of singularities problem. The first one, given by Havin \cite{Ha1} (see also \cite{Ai} or \cite{MK}), is based on a very beautiful duality argument. 
The second one, given by Müller and Wengenroth \cite[Theorem 1]{MW}, uses the open mapping theorem and Roth's fusion lemma, and establishes a link between this problem and approximation theory. 

A related question 
is to know whether there exists a bounded linear operator $T:\Hol(\Omega_1\cap\Omega_2)\to\Hol(\Omega_1)\times \Hol(\Omega_2)$, $f \mapsto (f_1, \, f_2)$, such that $f=f_1+f_2$. Mityagin and Khenkin \cite{MK} proved that such an operator does not always exist. 
\\

The problem has attracted a lot of interest in particular in Banach spaces of analytic functions. A challenging situation is the separation of singularities problem in the space $H^\infty$ of bounded analytic functions, which arises naturally in connection with interpolation problems \cite{Pol,PK}.
Havin, Nersessian \cite{HN}, Havin \cite{Ha2}, and Havin, Nersessian, Ortega-Cerdá \cite{HNO} solved it in several general configurations. Unlike the classical problem, they proved also that the problem has not a positive solution 
for arbitrary pairs of open sets, giving a lot of instructive counterexamples. The authors used an explicit Cauchy integral approach in the first two papers cited above and a reduction to the $\overline{\partial}$-equation (as in the modern solution of the classical problem) in the last paper. They constructed bounded linear separation operators explicitly in both cases.

Another interesting situation previously studied concerns Smirnov spaces. Assume that $\Omega$ is a simply connected domain in the complex plane with at least two boundary points. We say that $f\in \Hol(\Omega)$ belongs to the Smirnov space $E^p\left(\Omega \right)$ $(0 < p < + \infty)$ if there exists a sequence $(\gamma_n)_{n \in \N}$ of rectifiable Jordan curves eventually surrounding each compact subset of $\Omega$ such that 
$$
 \|f\|_{E^p}^p=\sup_{n \in \N} \int_{\gamma_n} |f(z)|^p |dz| < \infty.
$$ 
For $1\leq p < + \infty$, $E^p(\Omega)$ is a Banach space.
We mention that when the conformal mapping from the unit disk $\D$ to $\Omega$ has
bounded and invertibly bounded derivative, then the Smirnov space is isomorphic to the
corresponding Hardy space (see \cite[Theorem 10.2]{Du}).

Aizenberg \cite[Theorem 2]{Ai} solved the problem for the Smirnov space $E^p$ ($1<p<+\infty$) in the case of the intersection of $k$ bounded domains with regular boundaries (Ahlfors-regularity). The proof relies heavily on a strong result by David \cite{Da}, who studies the boundedness of the Hilbert transform on such regular curves. It can easily be generalized to finitely multiply connected domains using the same argument (see \cite[p. 182]{Du} for the definition of Smirnov space on a finitely connected domain). 
He also gave the same kind of result (with more regularity hypotheses) for the Hardy space in several complex variables (see \cite[Theorem 9]{Ai} for the definition of this space and the theorem). 

M\"{u}ller and Wengenroth \cite[Theorem 3]{MW} proved that solving the problem on the space $\mathcal{A}$ of holomorphic functions which are continuous up to the boundary is equivalent to prove a Roth fusion type lemma. In the same vein, Kaufman \cite[section 16.18]{HNi} asked for a solution to the problem on the space $\mathcal{A}^{(n)}$ of functions which have their first $n$ derivatives in $\mathcal{A}$. He mentionned that a positive answer for any $n \geq 1$ would provide information on the triviality of $\mathcal{A}^{(n)}(\Omega_1 \cap \Omega_2)$, where  triviality means $\mathcal{A}^{(n)}(\Omega_1 \cap \Omega_2) = \mathrm{Hol}(\C)_{|\Omega_1 \cap \Omega_2}$.
\\

The aim of this paper is two-fold. We  first discuss the problem of separation of singularities  for another prominent space of holomorphic functions, namely the Bergman space, and in particular on convex polygons (but not only). This is a very natural problem since, besides Hardy, Dirichlet and Fock spaces, the Bergman space is a central space in complex analysis bearing still a lof of challenging problems. The second aspect, and which was a central motiviation of this work, is that the Bergman space --- in particular on a square --- appears to be a keystone for identifying the reachable states of the 1-D heat equation on a finite interval with boundary controls. Characterizing these reachable states is a very prominent problem in control theory and has captivated a lot of research efforts since the groundbraking work of Fattorini and Russell in the early 70's \cite{FR}  and who showed that the reachable states in this setting extend to holomorphic functions on a square the diagonal of which is the interval on which the heat equation is controlled. More recently, an accelerating activity has taken place in order to better understand the exact nature of these holomorphic functions (see for instance\cite{Sc,MRR,DE}), culminating in the paper \cite{HKT} where it was shown that the reachable states are sandwiched between two well known Hilbert spaces of holomorphic functions: the Smirnov space (which is a companion space to the Hardy space) and the Bergman space on the square (more detailled information on this problem will be given below). It was also conjectured in that paper that the set of reachable states ``is not far of coinciding with'' this Bergman space \cite[Remark 1.3]{HKT}. Another key result in this connection was established very recently in \cite{Or} who showed that the reachable states are exactly given by the sum of two Bergman spaces on suitable sectors (see also \cite{KNT} for another proof given subsequently of this result), and which establishes the connection with the purely complex analytic problem of separation of singularities. 
In view of the result \cite{Or} mentioned above, our result applied in a very simple situation provides an affirmative answer to the conjecture raised in \cite[Remark 1.3]{HKT} and thus the definite characterization of the reachable states of the 1-D heat equation on a finite interval with boundary controls.

\subsection{Results on separation of singularities in Bergman spaces}

We will now turn to our first set of results concerning separation of singularities in Bergman spaces. Note that this problem is mentioned explicitely in \cite[p.17]{BKN} without an exact reference.

Let us begin with the definition and some well-known facts. Let $\omega$ be a non-negative mesurable function on $\Omega$. For $1 \leq p < \infty$, the weighted Bergman space $A^p\left(\Omega, \omega \right)$ consists of all functions $f \in \Hol(\Omega)$ such that $$
\|f\|_{A^p(\Omega,\omega)}^p:=\int_{\Omega} |f(x+iy)|^p \omega(x+iy)dxdy < +\infty. 
$$
When $\omega =1$, $A^p(\Omega,\omega)$ is the classical Bergman space which we simply denote by $A^p\left(\Omega \right)$.
When $p=2$, $A^2(\Omega)$ is a Reproducing Kernel Hilbert Space (RKHS) and we denote by $k_\lambda^\Omega$ its reproducing kernel, i.e. for any $f \in A^2(\Omega)$ and any $\lambda \in \Omega$, $f(\lambda)={\prodscal{f}{k_\lambda^\Omega}}_{L^2(\Omega)}$. Clearly, $A^2(\Omega)$ is a closed subspace of $L^2(\Omega)$, and the corresponding orthogonal projection, called \emph{Bergman projection}, is given by 
\begin{equation}
\label{proj}
\left(\mathrm{P}_{\Omega}f\right)(\lambda)= {\prodscal{f}{ k_\lambda^\Omega}}_{L^2(\Omega)}, \quad \lambda \in \Omega.
\end{equation}
The most prominent case is when $\Omega$ is the unit disk  in the complex plane $\D=\{z\in\C: |z|<1\}$ and $\omega=1$. Then for $\lambda\in\D$ (see e.g.\ \cite[Section 1.2]{DS}),
\begin{equation}\label{repkerD}
 k^{\D}_{\lambda}(z)=\frac{1}{\pi(1-\overline{\lambda}z)^2},\quad z\in\D.
\end{equation}

Also, if $\Omega_1$ and $\Omega_2$ are two open subsets of $\C$, and $\varphi : \Omega_1 \to \Omega_2$ is a conformal mapping, it is well-known (see \cite[Chapter VIII, Theorem 4.9, p.280]{QQ} or \cite[Chapter 1, § 1.3, Theorem 3]{DS}) that we have the following conformal invariance property 
\begin{equation}
\label{confmap}
k^{\Omega_1}_\lambda(z)= k^{\Omega_2}_{\varphi(\lambda)}(\varphi(z)) \varphi'(z) \overline{\varphi'(\lambda)} 
\end{equation}


The following weight will play a central r\^ole in our study: for  $N \in \N$, we write 
\begin{equation}
\label{weight}
\omega_N(z)= (1+\abs{z}^{2p})^{-N}.
\end{equation}

We start with a quasi separation of singularities theorem, in the sense that we have to add a weight with decay at infinity. It deals with general open sets $\Omega_1$ and $\Omega_2$ of $\C$ such that $\Omega_1 \setminus \Omega_2$ and $\Omega_2\setminus \Omega_1$ are far. Note that this condition already appears in \cite[Cor. 3.3]{HN} as an easy case for solving the separation of singularities problem in $H^{\infty}$.

\begin{thm}
\label{thmCousin}
Let $1 < p < \infty$.
Let $\Omega_1$ and $\Omega_2$ be open sets of $\C$ such that $\Omega_1 \cap \Omega_2 \neq \emptyset$. If $\dist(\Omega_1 \setminus \Omega_2,\  \Omega_2 \setminus \Omega_1) >0$, then we have $A^p(\Omega_1 \cap \Omega_2) \subset A^p(\Omega_1, \omega_1) + A^p(\Omega_2,\omega_1)$. 
\end{thm}

The previous theorem is based on a reduction to the $\bar \partial$-equation, as in the modern solution of the classical problem (see \cite[Thm 1.4.5]{Ho2} or \cite[Thm 9.4.1] {AM}), and on Hörmander type $L^p$-estimates for the $\bar \partial$-equation. This method was already used in \cite[Thm 1.2]{Or} to prove another kind of weighted separation theorem (see Corollary \ref{reach} below for an improvement of this theorem). Using the fact that polynomials not vanishing on $\overline{\Omega}$ are invertible multipliers of the Bergman space on a bounded domain, 
will allow us to show our first ``real'' (i.e. unweighted) separation result for bounded intersections. 
\begin{cor}
\label{corCousin}
Under the same hypotheses as in Theorem \ref{thmCousin}, if in addition $\Omega_1 \cap \Omega_2$ is bounded and $\overline{\Omega_1 \cup \Omega_2} \neq \C$, then $A^p(\Omega_1 \cap \Omega_2)= A^p(\Omega_1) + A^p(\Omega_2)$. 
\end{cor}

The case when $\dist(\Omega_1 \setminus \Omega_2,\  \Omega_2 \setminus \Omega_1) =0$
is more intricate. 
Let us begin with the simplest configuration of interest for us: $\Omega_1$ and $\Omega_2$ are half planes which intersect perpendicularly. By rotation and translation we can of course reduce the situation to the upper and right half planes: $\C^+=\enstq{z\in \C}{\mathrm{Im}(z) >0}$ and $\C_+=\enstq{z\in \C}{\mathrm{Re}(z) >0}$. We write $\C^{++}=\C_+ \cap \C^+$ for the resulting quarter plane. 

\begin{thm}
\label{thm1}
Let $1<p<\infty$. 
Then $A^p(\C^+)  + A^p(\C_+) = A^p(\C^{++})$. 
\end{thm}
The proof of this theorem is strikingly simple when $p=2$ where it uses only the explicit form of the reproducing kernels of the two half planes. Though the same idea does not apply to arbitrary sectors we can reduce that general situation to right angle sectors which leads to our next result.

\begin{thm}
\label{thm2}
Let $1 < p < \infty$.
Let $H_1, H_2$ be two half planes such that $\Sigma:= H_1\cap H_2 \neq \emptyset$ is a sector. 
Then $A^p(\Sigma)= A^p(H_1) + A^p(H_2)$. 
\end{thm}

The main result of this part of the paper is the separation of singularities problem for $n$ half planes, the intersection of which is a convex polygon. 
\begin{thm}
\label{thm3}
Let $1 < p < \infty$.
Let $H_1, H_2, \dots, H_n$ be half planes such that $\mathcal{P}:= \bigcap_{k=1}^n H_k \neq \emptyset$ is a convex polygon. 
Then $A^p(\mathcal{P})= \sum_{k=1}^n A^p(H_k)$. 
\end{thm}

It is worth mentioning that when $\Omega$ is a polygon it is known that Schwarz-Christoffel mappings allow to send the upper half plane conformally onto $\Omega$, so that with \eqref{confmap} it is possible to define the Bergman kernel for $A^2(\Omega)$. However, already for a square, the understanding of the corresponding reproducing kernel is a very non-trivial matter.
\\

Let us consider a special case illustrating the above results:
$$
 \Omega=\{z=x+iy\in \C:0<x<1, 0<y<1\}
$$
i.e. $\Omega$
is the unit square with sides parallel to the coordinate axes and lower left corner $0$. Let $\Omega_1=\C^{++}$ and $\Omega_2=(1+i)-\C^{++}$. Then
$\Omega=\Omega_1\cap\Omega_2$, and Theorems \ref{thm2} and \ref{thm3} yield the 
following immediate consequence which will resolve the conjecture on the reachable states of the 1-D heat equation on a finite rod with boundary controls as discussed in the next section.

\begin{cor}\label{coro1}
We have $A^p(\Omega)=A^p(\C^{++})+A^p((1+i)-\C^{++})$.
\end{cor}


It turns out that we can apply Theorem \ref{thm3} to more general domains. More precisely,
we will consider non-empty, bounded intersections of convex domains $\Omega_1$ and
$\Omega_2$. Then the boundaries $\partial\Omega_1$ and $\partial\Omega_2$ can meet in single points or along curves. We will assume that there are only finitely many single points and arcs, i.e.\ $\tilde\partial(\partial\Omega_1\cap \partial\Omega_2)$ is finite (by $\partial\Omega$ we mean the boundary of a two dimensional manifold $\Omega$, and by $\tilde\partial E$ the boundary of a one-dimensional manifold $E$).


\begin{thm}
\label{thmcvx}
Let $1 < p < \infty$.
Let $\Omega_1$ and $\Omega_2$ be two open convex sets in $\C$ such that 
\begin{enumerate}[label=(\roman*)]
\item $\Omega_1 \cap \Omega_2$ is non-empty and bounded,
\item The set $\tilde{\partial} \left(\partial \Omega_1 \cap \partial \Omega_2 \right)$ is finite.  
\end{enumerate}
Then $A^p(\Omega_1 \cap \Omega_2)= A^p(\Omega_1) + A^p(\Omega_2)$. 
\end{thm}

All these theorems have weighted versions with weights $\omega_l$ ($l \in \N$). 
Let us mention another direct consequence. Recall that the Dirichlet space $\mathcal{D}(\Omega)$ consists of all functions $f$ holomorphic on $\Omega$ satisfying $f' \in A^2(\Omega)$ (since the formerly stated results for Bergman spaces work for $1<p<\infty$, we can also consider the corresponding Dirichet type spaces which are rather called Besov spaces). 
Applying the above decompositions to $f'$ and taking anti-derivatives yields the corresponding decompositions in Dirichlet spaces. Note the following general results. 
\begin{prop}
\label{Dirichlet}
Let $\Omega_1$ and $\Omega_2$ be two simply connected domains in $\C$ such that $\Omega_1 \cap \Omega_2 \neq \emptyset$. 
If $A^2(\Omega_1 \cap \Omega_2) = A^2(\Omega_1) +A^2(\Omega_2)$ then $\mathcal{D}(\Omega_1 \cap \Omega_2) = \mathcal{D}(\Omega_1) +\mathcal{D}(\Omega_2)$. 
\end{prop}
An application of this observation solves the control problem of the heat equation with Neumann boundary control.
\\

Finally, we emphasize that our proofs do not work for $p=1$. This situation already occurs in Aizenberg's result for the Smirnov space \cite{Ai}. While in his work it is the failure of boundedness of the Riesz projection which makes obstruction, here it is the Bergman projection which is not bounded on $L^1$. This leads to the following open question.
\\

\noindent \textbf{Question:} Is there a positive solution to the separation of singularities problem in $E^1$ and $A^1$?

\subsection{The reachable states of the 1-D heat equation}

Let us now turn to an important application of the separation of singularities in Bergman spaces: the control of the 1-D heat equation with boundary controls. 
More precisely, we consider the heat equation on the segment $[0, \, \pi]$ with Dirichlet boundary control at both ends. 
\begin{equation}
\label{HE}
\tag{HE}
\left\lbrace
	\begin{aligned}
		&\frac{\partial y}{\partial t}(t,x)-		\frac{\partial^2 y}{\partial x^2} = 0  \qquad &t >0, \ x\in (0,\pi),& \\
		&y(t,0)=u_0(t),\  \ y(t,\pi)=u_\pi(t) \qquad & t >0, &\\
		&y(0,x)= f(x) \qquad &x \in (0, \pi),
	\end{aligned}
\right.
\end{equation}
Let $(\T_t)_{t\geq 0}$ be the semigroup generated by the Dirichlet Laplacian on $(0, \pi)$.  
For any $u\!:=\! (u_0, u_\pi) \in L^2_{\mathrm{loc}}((0, +\infty), \C^2)$ --- the so-called input (or control) function --- and $f \in  X:=W^{-1,2}(0, \, \pi)$ (the dual of the Sobolev space $W_0^{1,2}(0, \, \pi)$), this equation admits a unique solution $y \in C\left((0, +\infty), X\right)$ (see \cite[Prop. 10.7.3]{TW}) defined by 
\begin{equation}
\label{sol}
\forall t >0, \ y(t, \cdot)= \T_tf + \Phi_t u
\end{equation}
where $\Phi_t \, \in \mathcal{L}(L^2([0,\tau], \C^2),X)$ is the controllability operator (see \cite[Prop. 4.2.5]{TW}). For $f \in X$ and $\tau >0$, we will say that $g \in X$ is reachable from $f$ in time $\tau$ if there exists a boundary control $u \in L^2((0,\tau), \C^2)$ such that the solution of $\eqref{HE}$ satisfies $y(\tau, \cdot)=g$. We denote by $\mathcal{R}^f_\tau$ the set of all reachable functions from $f$ in time $\tau$. 
\\

{\bf Main question: } Can we describe explicitely the set $\mathcal{R}^f_\tau$? \\

There is a large literature on this problem that we shall recall briefly here. 
We first mention that $\mathcal{R}^f_\tau$ does not depend on the initial condition $f \in X$, which means by \eqref{sol} that $\mathcal{R}^f_\tau$ is the linear space $\Ran \Phi_\tau$. So, it will be denoted by $\Ran \Phi_\tau$ from now on and called \emph{reachable space}. Secondly, it does not depend on time (see \cite{Fa}, \cite{Se}, or \cite[Remark 1.1]{HKT}). Finally, the functions in $\Ran \Phi_\tau$ can be extended analytically to the square $D=\enstq{z=x+iy \in \C}{\abs{x-\frac{\pi}{2}} + \abs{y} < \frac{\pi}{2}}$. 
The description of the reachable space started in the pioneering work \cite{FR} and has been improved successively in \cite{Sc}, \cite[Thm 1.1]{MRR}, \cite[Theorem 1.1]{DE}, \cite[Proposition 1.1, Thm 1.2]{HKT}, \cite{Or}; see also \cite{KNT} for a recent survey on this problem and some other related problems. 
The key result obtained in \cite{Or} is the following (see also \cite[Section 7]{KNT} for another proof given subsequently of this result).
\begin{thm}
\label{Or-result}
Let $\Delta=\enstq{z \in \C}{\abs{\arg{(z)}} < \pi/4}$.\\
We have $\mathrm{Ran}\Phi_\tau = A^2(\Delta) + A^2(\pi-\Delta)$.
\end{thm}

By rotation and rescaling, Corollary \ref{coro1} immediately yields
\begin{equation}
\label{eq-square-decomp}
 A^2(D)=A^2(\Delta) + A^2(\pi-\Delta),
\end{equation}
which gives the following characterization of the reachable space. 
\begin{cor}
\label{reach}
We have $\mathrm{Ran} \Phi_\tau = A^2(D)$.
\end{cor}
This proves the conjecture stated in \cite[Remark 1.3]{HKT}.
This implies also obviously the following inclusion.
\begin{cor}
\label{embedding}
We have $A^2(D) \subset X:=W^{-1,2}(0, \pi)$. 
\end{cor}

The paper is organized as follows. In Section \ref{sec2} we prove the separation of singularities results, and in Section \ref{sec3} we give a more transparent proof to Corollary \ref{coro1} and apply the results to several related problems on reachable spaces of the heat equation. 

\section{Proof of theorems \label{sec2}}
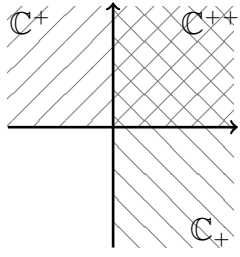
\begin{figure}[h]
\begin{center}
\begin{tikzpicture}[scale=0.2]


\fill[line space=10pt, pattern=my north west lines, pattern color=gray] (0,8) -- (8,8) -- (8,-8) -- (0,-8) -- cycle ;
\fill[line space=10pt, pattern=my north east lines, pattern color=gray] (-7,0.) -- (8,0) -- (8,8) -- (-7,8) -- cycle ;

\draw  [line width=1.pt,->](-7,0)-- (8.3,0);
\draw  [line width=1.pt,->](0,-8)-- (0,8.3);

\draw[color=black] (-5.5,7) node[xscale=1, yscale=1] {$\C^{+}$};
\draw[color=black] (6.5,7) node[xscale=1, yscale=1] {$\C^{++}$};
\draw[color=black] (6.5,-7) node[xscale=1, yscale=1] {$\C_{+}$};
\end{tikzpicture}
\caption{\label{figure1} The half planes $\C^{+}$ and $\C_{+}$, and their intersection, the quarter plane $\C^{++}$.}
\end{center}
\end{figure}
\begin{proof}[Proof of Theorem \ref{thm1}]

Obviously we only have to show the reverse inclusion. So, let us start with $f \in A^2(\C^{++})$. 
Using the conformal invariance property \eqref{confmap} applied to the kernel on $\D$ introduced in \eqref{repkerD} and using the conformal map $\varphi:\C^+\to \D$,
$$
 \varphi(z)=\frac{z-i}{z+i},
$$ 
we obtain first the  reproducing kernel on $\C^+$,
\begin{equation*}
k_\lambda^{\C^+}(z) = \frac{-1}{\pi (z-\bar{\lambda})^2}, \quad\lambda, z \in \C^+.
\end{equation*}
The kernel for $\C_+$ is deduced from this just by a suitable rotation
\begin{equation*}
k_\lambda^{\C_+}(z) = \frac{1}{\pi (z+\bar{\lambda})^2}, \quad \lambda, z \in \C_+ .
\end{equation*}
Finally, for the kernel on the quarter plane $\C^{++}$, use $\varphi:\C^{++}\to C_+$, 
$\varphi(z)=z^2$ to get
\begin{equation*}
k_\lambda^{\C^{++}}(z) = \frac{-4z \bar{\lambda}}{\pi (z^2-\bar{\lambda}^2)^2}, \quad\lambda, z \in \C^{++} 
\end{equation*} 
An easy computation leads to the following key observation:
\begin{equation}
\label{reproducing}
k_\lambda^{\C^+} + k_\lambda^{\C_+} = k_\lambda^{\C^{++}}, \qquad \forall \lambda \in \C^{++}.
\end{equation}

For a function $f$ defined on $\C^{++}$, we write $S_{\C^+} f$ (resp. $S_{\C_+} f$) for the trivial extension of $f$ by $0$ on $\C^+$ (resp. $\C_+$) outside $\C^{++}$, i.e. 
$$S_{\C^+} f(z)=
\begin{cases}
f(z) \qquad \text{if } z \in \C^{++} \\
0  \qquad \text{if } z \in \C^+ \setminus \C^{++} 
\end{cases}$$
and correspondingly for $S_{\C_+}f$. 
Hence, since $f$ was assumed in $A^2(\C^{++})$, for every $\lambda \in \C^{++}$, we have 
\begin{align*}
f(\lambda)={\prodscal{f}{ k_\lambda^{\C^{++}}}}_{L^2(\C^{++})} &=\int_{\C^{++}}f(z)\overline{ k_\lambda^{\C^{++}}(z)} dA(z) \\
&= \int_{\C^{++}}f(z) \overline{ \left(k_\lambda^{\C^+}(z) + k_\lambda^{\C_+}(z)\right)} dA(z) \\
&= \int_{\C^+}S_{\C^+} f(z) \overline{ k_\lambda^{\C^+}(z)} dA(z) + \int_{\C_+}S_{\C_+} f(z) \overline{ k_\lambda^{\C_+}(z)} dA(z) \\
&={\prodscal{S_{\C^+} f}{ k_\lambda^{\C^+}}}_{L^2(\C^+)} + {\prodscal{S_{\C_+} f}{ k_\lambda^{\C_+}}}_{L^2(\C_+)}.
\end{align*}
Finally, using the Bergman projection introduced in \eqref{proj}, we obtain on $\C^{++}$
\begin{equation}
\label{eq}
f = \mathrm{P}_{\C^+}\left(S_{\C^+} f\right) + \mathrm{P}_{\C_+} \left(S_{\C_+} f\right) \ \in \  A^2(\C^+)+A^2(\C_+).
\end{equation}
The result follows. 

Consider the case $p\neq 2$. Pick $f \in A^p(\C^{++})$. 
Since $k_\lambda^{\C^+}$ (resp. $k_\lambda^{\C_+}$) belongs to $L^q(\C^+)$ (resp. $L^q(\C_+)$) for $q>1$, the right hand side is well-defined for $f \in L^p(\C^{++})$ if $1 \leq p < \infty$. 
In addition, it is well-known (see \cite[Thm. 1.34]{BBGNPR}) that $\mathrm{P}_{\C^+}$ (resp. $\mathrm{P}_{\C_+}$) is bounded from $L^p(\C^+)$ (resp. $L^p(\C_+)$) onto $A^p(\C^+)$ (resp. $A^p(\C_+)$) if and only if $p>1$. Finally, equality \eqref{eq} holds for all $f \in A^2(\C^{++}) \cap A^p(\C^{++})$ and this last space is dense in $A^p(\C^{++})$ (see Remark \ref{rmkdensity} below), hence by continuity it holds also for every $f \in A^p(\C^{++})$. The proof is complete.
\end{proof}

It should be pointed out that the above argument yields a linear bounded separation operator.

\begin{rmk}
\label{rmkdensity}
For a general open set $\Omega \subset \C$, the density of $A^2(\Omega) \cap A^p(\Omega)$ in $A^p(\Omega)$ is a difficult problem (see \cite[Proposition 2.2]{He}) but in our specific case it follows from \cite[Proposition 1.17]{BBGNPR} who showed the result for $\Omega=\C^{+}$. Indeed, moving back and forth between $A^p(\C^+)$ 
and $A^p(\C^{++})$ {\it via} the change of variables fomula $T_pF(z)=z^{2/p}F(z^2)$, $F\in A^p(\C^+)$, $z\in \C^{++}$, will produce the desired result. More precisely, in order to approximate a function $F\in A^p(\C^+)$ by $A^2(\C^+)$-functions the authors of \cite{BBGNPR} introduce the following regularization of $F$ by shifting and multiplying with a suitable function: $F_{\varepsilon,\alpha}(z)=F(z+i\varepsilon)G_{\alpha}(\varepsilon z)$, where $G_{\alpha}(z)=(1-iz)^{-(2+\alpha)}$, $\alpha\ge 0$ and $\varepsilon>0$. Clearly $F_{\varepsilon,\alpha}(z)\to F(z)$, when $\varepsilon\to 0$, for every $z\in \C^+$. As observed in \cite[Proposition 1.3]{BBGNPR}, the function $F(z+i\varepsilon)$ is in the Hardy space of the upper half plane $H^p(\C^+)$ which allows an application of the dominated convergence theorem (actually to horizontal $p$-means of $F_{\varepsilon,\alpha}$) when $\varepsilon\to 0$, which yields that $F_{\varepsilon,\alpha}\to F$ in $A^p(\C^+)$. Consequently $f_{\epsilon,\alpha}:=T_pF_{\epsilon,\alpha}\to f$ in $A^p(\C^{++})$. It remains to prove that $f_{\varepsilon,\alpha}\in A^2(\C^{++})$, i.e. $$ f_{\varepsilon,\alpha}=T_pF_{\epsilon,\alpha} = z^{2/p}F(z^2+i\varepsilon)G_{\alpha}(\varepsilon z^2)= \frac{z^{2/p-1}}{(1-i\varepsilon z^2)^{\alpha}} T_2F_{\varepsilon,0}(z)$$ is in $A^2(\C^{++})$. Since $F_{\varepsilon,0}\le C_{\varepsilon}/(1+|z|)^2$, we have $F_{\varepsilon,0}\in A^2(\C^+)$, and so $f_{\varepsilon,\alpha}\in A^2(\C^{++})$ can now be reached by an appropriate choice of $\alpha$ depending on $p$ (note that $T_2F_{\varepsilon,0}$ is locally bounded at $0$, so that only the case $1<p<2$ needs consideration of a suitable $\alpha$).
\end{rmk}
 
\begin{rmk}
\label{rmk1}
Obviously, the theorem holds for every half plane the intersection of which is a right-angle sector. \end{rmk}

We will now move on to the proof of Theorem \ref{thm2}. 
For $-\pi \leq a < b \leq \pi$, we denote by $\Delta_a^b$ the angular sector $\Delta_a^b= \enstq{z \in \C}{a < \arg(z) <b}$. 

\begin{figure}[h]
\begin{center}
\begin{tikzpicture}[scale=0.7]

\fill[line space=10pt, pattern=my north west lines, pattern color=gray] (-2,4) -- (0,0) -- (6,2) -- (6,4) -- cycle ;
\draw  [line width=1.pt](0,0)-- (6,2);
\draw  [line width=1.pt](0,0) -- (-2,4);


\draw  [line width=1.pt,->](-2,0)-- (6.2,0);
\draw  [line width=1.pt,->](0,-3)-- (0,5);

\draw[color=black] (5,2.2) node[xscale=1, yscale=1] {$\Delta_a^{b}$};
\draw[color=black] (1.5,0.2) node[xscale=1, yscale=1] {$a$};
\draw[color=black] (1.7,2) node[xscale=1, yscale=1] {${b}$};
\draw[color=black](2,0) arc(0:18:2);
\draw[color=black](3,0) arc(0:117:3);
\end{tikzpicture}
\caption{\label{figsecteurs} The sector $\Delta_a^b$.}
\end{center}
\end{figure}
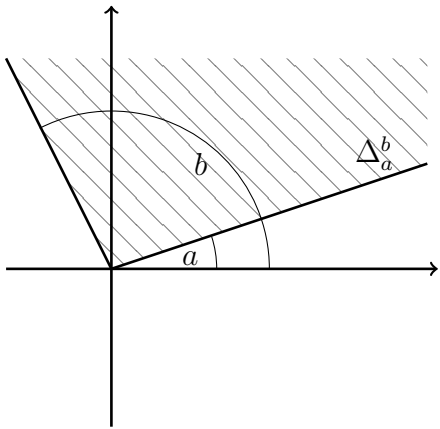

We can of course reduce the situation to the case $a=0$, and consider $\Delta_0^b=\C^+\cap H_1$ where $H_1$ is the half plane $\Delta_{b-\pi}^b$. 
While it is very tempting now to apply the same idea above to an arbitrary sector $\Delta_0^b$, the magic decomposition of the reproducing kernel on the right-angle sector breaks down. Indeed, with formula \eqref{confmap} in mind one can of course explicitely compute the kernel for $H_1$ which amounts essentially to multiply $z$ and $\lambda$ by a suitable unimodular constant $\alpha$ (more precisely $\alpha=e^{i(\pi-b)}$) in the expression of $k_\lambda^{\C^+}(z)$. However, the same transformation formula \eqref{confmap} applied to transform the kernel of $\C^{+}$ to that of $\Delta_0^b$ involves a power function: $\varphi(z)=z^{\pi/b}$. A computation shows that the kernels of the half planes do not add up to the kernel of the sector.


\begin{proof}[Proof of Theorem \ref{thm2}]
As already mentioned it is enough to prove the result for $\Sigma = \Delta_0^b$, $H_1=\C^+$ and $H_2=\C_\theta:=\Delta_{-\pi+b}^b$, with $0 < b < \pi$: 
\begin{eqnarray}\label{eq7}
 A^p(\Delta_0^{b})=A^p(\C^+)+A^p(\Delta_{-\pi+b}^{b}).
\end{eqnarray}
The heart of the proof is contained in the following lemma which shows in a way that we can double the opening of the sector. 

\begin{lem}\label{DecompLemma}
Let $a, b$ be real numbers such that $-\pi \leq a < b \leq \pi$.
Then $A^p(\Delta_a^b)=A^p\left(\Delta_{a}^{\min(\pi, \, 2b-a)}\right) + A^p\left(\Delta_{\max(-\pi, \, 2a-b)}^{b}\right)$.
\end{lem} 

\begin{proof}
Let $\varphi : \Delta_a^b \to \C^{++}$ be the conformal mapping given by $\varphi(z)= (e^{-ia}z)^{\frac{\pi}{2(b-a)}}$ where we have chosen the branch cut to be $(-\infty, \, 0]$. 
Let $T : A^p(\C^{++}) \to A^p(\Delta_a^b)$ be the isometric isomorphism associated with $\varphi$, i.e 
$$\forall g \in A^p(\C^{++}), \ Tg=(g \circ \varphi) (\varphi')^{2/p}.$$
Pick $f \in A^p(\Delta_a^b)$. Then $g:= T^{-1}f$ belongs to $A^p(\C^{++})$. So, by Theorem \ref{thm1}, there exist $g_1 \in A^p(\C^+)$ and $g_2 \in A^p(\C_+)$ such that $g={g_1}+ {g_2}$ on $\C^{++}$. Hence 
$$
 f=Tg=Tg_1+Tg_2=(g_1 \circ \varphi) (\varphi')^{2/p} + (g_2 \circ \varphi) (\varphi')^{2/p}.
$$ 
Remark now that the branch cut has been choosen such that $\varphi$ continues analytically on $\Delta_{a}^{\min(\pi, \, 2b-a)}$ and $\Delta_{\max(-\pi, \, 2a-b)}^{b}$. Moreover, $\varphi(\Delta_{a}^{\min(\pi, \, 2b-a)}) \subset \C^+$ and $\varphi(\Delta_{\max(-\pi, \, 2a-b)}^{b}) \subset \C_+$, so that $g_1\circ \varphi$ and $g_2\circ\varphi$ are well defined holomorphic functions. Thus, $f \in A^p\left(\Delta_{a}^{\min(\pi, \, 2b-a)}\right) + A^p\left(\Delta_{\max(-\pi, \, 2a-b)}^{b}\right)$, which proves the lemma.
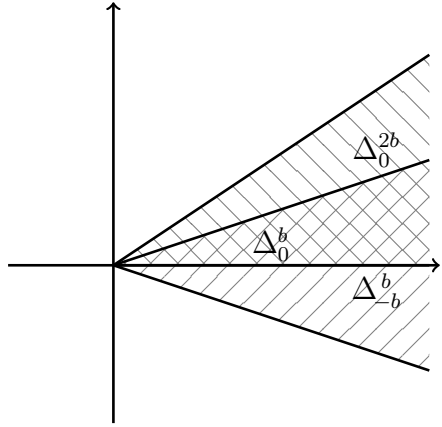
\begin{figure}[h]
\begin{center}
\begin{tikzpicture}[scale=0.7]

\fill[line space=10pt, pattern=my north west lines, pattern color=gray] (0,0) -- (6,4) -- (6,0)--  cycle ;
\fill[line space=10pt, pattern=my north east lines, pattern color=gray] (6,2) -- (0,0) -- (6,-2) -- cycle ;
\draw  [line width=1.pt](0,0)-- (6,2);
\draw  [line width=1.pt](0,0) -- (6,4);
\draw  [line width=1.pt](0,0) -- (6,-2);
\draw  [line width=1.pt](0,0) -- (6,0);


\draw  [line width=1.pt,->](-2,0)-- (6.2,0);
\draw  [line width=1.pt,->](0,-3)-- (0,5);

\draw[color=black] (3,0.4) node[xscale=1, yscale=1] {$\Delta_0^b$};
\draw[color=black] (5,2.2) node[xscale=1, yscale=1] {$\Delta_0^{2b}$};
\draw[color=black] (5,-0.5) node[xscale=1, yscale=1] {$\Delta_{\!-b}^{\, b}$};
\end{tikzpicture}
\caption{\label{figsecteurs2} One step of the induction.}
\end{center}
\end{figure}
\end{proof}

We are now in a position to prove the theorem. We start from $A^p(\Delta_0^\theta)$, where we assume for the moment that $\pi/2<\theta<\pi$. 
Since $\min(\pi,2\theta)=\pi$, the lemma yields
$$
 A^p(\Delta_0^{\theta})=A^p(\Delta_0^{\pi})+A^p(\Delta_{-\theta}^{\theta}) 
$$
so that $f\in A^p(\Delta_0^{\theta})$ decomposes as $f=f_1+f_2$ (considered on $\Delta_0^{\theta}$) with $f_1\in A^p(\Delta_0^{\pi})=A^p(\C^+)$ and $f_2\in A^p(\Delta_{-\theta}^{\theta})$. Since $\pi/2<\theta$ we have $\Delta_{-\theta}^{\theta}\supset \Delta_{\theta-\pi}^{\theta}$ which yields \eqref{eq7}.

We will proceed by an inductive application of the lemma.
In order to better understand this induction, let us also illustrate the case when
$\pi/4<\theta<\pi/2$. In order to not overcharge notation we will only mention the underlying sectors and not write out the Bergman spaces, see Figure \ref{decompSectors}.

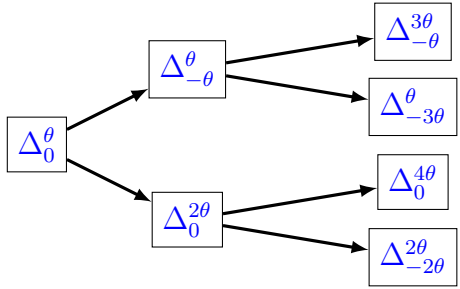
\begin{figure}[h]
\begin{center}
\begin{tikzpicture}
\tikzstyle{quadri}=[rectangle,draw,text=blue]
 \tikzstyle{estun}=[->,>=latex,very thick]
 \node[quadri] (13) at (-2,0) { $\Delta_0^{\theta}$ };
 \node[quadri] (X11) at (0,1) {$\Delta_{-\theta}^{\theta}$};
 \node[quadri] (X12) at (0,-1) {$\Delta_{0}^{2\theta}$};
 \node[quadri] (X21) at (3,1.5) {$\Delta_{-\theta}^{3\theta}$};
 \node[quadri] (X22a) at (3,0.5) {$\Delta_{-3\theta}^{\theta}$};
 \node[quadri] (X22b) at (3,-0.5) {$\Delta_{0}^{4\theta}$};
 \node[quadri] (X23) at (3,-1.5) {$\Delta_{-2\theta}^{2\theta}$};
  \draw[estun] (13)--(X11)node[midway,above]{};
  \draw[estun] (13)--(X12)node[midway,below]{};  
  \draw[estun] (X11)--(X21)node[midway,above]{};
  \draw[estun] (X11)--(X22a)node[midway,below]{};  
   \draw[estun] (X12)--(X22b)node[midway,above]{};
  \draw[estun] (X12)--(X23)node[midway,below]{}; 
%
%
\end{tikzpicture}
\caption{\label{decompSectors} Iterrative applications of Lemma \ref{DecompLemma} in the decomposition of the sector $\Delta_0^{\theta}$.}
\end{center}
\end{figure}

It is clear from here that after $n$ steps of applications of the lemma, there are $2^n$
sectors $\Delta_{-k\theta}^{(2^n-k)\theta}$, $k=0,\ldots, 2^n-1$ (when $(2^n-k)\theta>\pi$ or $k\theta>\pi$ it should be replaced by $\pi$).

In the general case, let $N \in \N^*$ be the least natural number  such that $2\pi < 2^N \theta $.
Observe that when $0\le k\le 2^{N-1}$ (which corresponds to one half of the possible $k$'s), then
$(2^N-k)\theta\ge (2^N-2^{N-1})\theta=2^{N-1}\theta>\pi$, so that $\Delta_{-k\theta}^{(2^N-k)
\theta}\supset \C^+$, while for $2^{N-1}<k\le 2^N-1$, we have $k\theta>\pi$ so that
$\Delta_{-k\theta}^{(2^N-k)\theta}\supset \Delta_{-\pi}^{\theta}\supset \Delta_{\theta-\pi}^{\theta}$. (We mention again that as soon as $(2^N-k)\theta>\pi$ or $k\theta>\pi$ in the procedure, it should be replaced by $\pi$.) 

Hence, any function $f\in A^p(\Delta_0^{\theta})$ will be decomposed into $2^N$ functions, one half of which is in $A^p(\C^+)$ and the other one in $A^p(\Delta_{-\pi+\theta}^{\theta})$.
\end{proof}

The next result claims that  Theorem \ref{thm2}  also holds for Bergman spaces with the same weight $\omega_N(z)=(1+|z|^{2p})^{-N}$ (see \eqref{weight}).
\begin{cor}
\label{weighted}
Let $H_1, H_2$ be two half planes such that $S_\theta:=H_1 \cap H_2 \neq \emptyset$ is an angular sector. For any $N \in \N$, we have 
$A^p(S_\theta, \omega_N)=A^p(H_1, \omega_N)  + A^p(H_2, \omega_N).$
\end{cor}

\begin{proof}
The direct inclusion is obvious, let us prove the converse one. 
Denote by $z_0$ a complex number such that $\dist(z_0,S_\theta)>0$ and $P(z)= (z-z_0)^{2N}$.
Pick $f \in A^p(S_\theta, \omega_N)$, then $\frac{f}{P}$ belongs to $A^p(S_\theta)$, and by Theorem \ref{thm2}, there exists $\tilde{f_1} \in A^p(H_1)$ and $\tilde{f_2} \in A^p(H_2)$ such that $\frac{f}{P} = \tilde{f_1} + \tilde{f_2}$. Hence, $f= P \tilde{f_1} +  P \tilde{f_2} := f_1 + f_2$ with $f_1 \in A^p(H_1, \omega_N)$ and $f_2 \in A^p(H_2, \omega_N)$. 
\end{proof}

We will now prove Theorem \ref{thmCousin} which is an almost-separation of singularities in the simplest case. Denote by $\bar\partial$ the Cauchy-Riemann operator $\bar\partial= \frac1{2}(\frac{d}{dx}+i \frac{d}{dy})$. 
The main idea is to reduce the problem to a $\overline{\partial}$-equation and to use Hörmander type $L^p$-estimates for the solution of the $\bar{\partial}$-equation. The estimates are certainly well-known to experts but we include a proof for completeness. 
\begin{lem}
\label{Hormander}
Let $1< p < \infty$.
Let $\Omega \subset \C$ be an open connected set such that $\overline{\Omega} \neq \C$. If $f \in L^p(\Omega)$
then there exists a solution $u$ of the equation $\bar{\partial} u = f$ on $\Omega$ such that $u \in L^p\left(\Omega, \, \omega_1 \right)$.
\end{lem}
Note that we do not look for a solution in $L^p(\Omega)$ but we allow a weight to appear which makes the problem solvable in the setting under consideration. Still, this solution will be sufficient for our purpose.
\begin{proof}
The case $p=2$ is a particular case of the famous Hörmander $L^2$-estimates \cite[Thm 4.2.1]{Ho1}. From now on, let $1 < p\neq2 < \infty$.

For bounded $\Omega$ the result can be found in \cite[Sec. 2, p.134]{FS} and follows from  Young's inequality and properties of the Cauchy kernel $\frac{1}{s-z}$. Indeed, the classical solution (i.e the solution with minimal weigthed $L^2$-norm) of $\bar\partial u = f$ on $\Omega$ is given by  
$$u(z)=\frac{1}{\pi}\int_\Omega \frac{f(s)}{s-z} dA(s) $$
and satisfies 
\begin{equation}
\label{eq-Young}
\|u\|_{L^p(\Omega)} \leq \|f\|_{L^p(\Omega)} \left\|\frac1{z}\right\|_{L^1(\Omega - \Omega)} \leq C \|f\|_{L^p(\Omega)}
\end{equation}
where $C$ depends only the diameter of $\Omega-\Omega=\{u-v:u\in\Omega, v\in\Omega\}$. 

For the general case, let $z_0\in\C$ be such that $\dist(z_0,\Omega)>0$ and set $Q(z) =z-z_0$. Denote by $g$ the function $g=\frac{f}{Q}$. Clearly,  $u_g$ satifies $\bar{\partial} u_g=g$ if and only if $\bar{\partial} u_f=f$, where $u_f= Q u_g$. It thus suffices to prove the existence of a solution $u_g$ to $\bar{\partial} u_g=g$ such that $u_g \in L^p\left(\Omega, (1+\abs{z}^{p})^{-1}\right)$. For this, we choose $u_g$ to be the classical solution of $\bar{\partial} u_g=g$ on $\Omega$ defined by 
$$u_g(z)= u_g^{\Omega}(z)=\frac{1}{\pi}\int_\Omega \frac{g(s)}{s-z} dA(s). $$

Let $\zeta\in\Omega$, then for every $z \in D(\zeta,\, 1)\cap\Omega$ we have 
$$u_g(z)= \int_{\Omega \setminus D(\zeta, \, 2)} \frac{g(s)}{s-z} dA(s) + \int_{\Omega \cap D(\zeta, \, 2)} \frac{g(s)}{s-z} dA(s):=u_g^{\Omega \setminus D(\zeta, \, 2)} + u_g^{\Omega \cap D(\zeta, \, 2)}. $$
Considering $u_g^{\Omega\cap D(\zeta,2)}$ as the solution to the $\bar{\partial}$-problem on the bounded domain $\Omega\cap D(\zeta,2)$, we obtain by \eqref{eq-Young}
\begin{align*}
\|u_g^{\Omega\cap D(\zeta,2)}\|_{L^p(\Omega\cap D(\zeta, \, 1))} &\leq \|u_g^{\Omega\cap D(\zeta,2)}\|_{L^p(\Omega\cap D(\zeta, \, 2))} \\
&\leq C \|g\|_{L^p(\Omega\cap D(\zeta, \, 2))} \\
&\leq \frac{C}{1 + |\zeta|} \|f\|_{L^p(\Omega\cap D(\zeta, \, 2))} \\
&\leq \frac{C}{1 + |\zeta|}
\end{align*} 
where $C$ is essentially given by $\left\|\frac1{z}\right\|_{L^1( D(\zeta, \, 2)-  D(\zeta, \, 2))}
\le  \left\|\frac1{z}\right\|_{L^1( D(0,\, 4))}$ and thus
independent on $\zeta$. 

For $u_g^{\Omega \setminus D(\zeta, \, 2)}$, Hölder's inequality gives for all $z \in D(\zeta, \, 1)$, 
$$ \abs{u_g^{\Omega \setminus D(\zeta, \, 2)}(z)} \leq {\|f\|}_{L^p \left(\Omega \setminus D(\zeta, \, 2)\right)}  {\left\|\frac{1}{(z-\cdot)Q}\right\|}_{L^{p'}\left(\Omega \setminus D(\zeta, \, 2)\right)} \leq 
\begin{cases}
\frac{C}{1+|\zeta|} \quad \text{if } p <2 \\
C \qquad \text{if } p>2
\end{cases} 
$$
where $p'$ is the conjugate exponent of $p$, and $C$ is independent on $\zeta$.

Putting the above estimates together, and observing that $1+|\zeta|\simeq 1+|z|$ for
$z\in D(\zeta,1)$, we get
$$
|u_g(z)|\le \frac{C}{1 + |\zeta|}+\begin{cases}
\frac{C}{1+|\zeta|} \quad \text{if } p <2 \\
C \qquad \text{if } p>2
\end{cases} 
 \le \begin{cases}
\frac{C'}{1+|z|} \quad \text{if } p <2 \\
C' \qquad \text{if } p>2
\end{cases} 
$$
(pick for instance $C'=2C$).

Hence
$$
\int_{\Omega} \frac{\abs{u_g(z)}^p}{1+\abs{z}^{p}} dA(z) 
\leq
\begin{cases}
\int_{\Omega} \frac{C'}{1+|z|^{2p}} dA(z) \qquad \text{if } p<2\\
\int_{\Omega} \frac{C'}{1+|z|^{p}} dA(z) \qquad \text{if } p>2
\end{cases}
< \infty.
$$
\end{proof}

Restating the previous lemma in the spirit of H\"ormander's result
is to say that the solution $u$ of $\overline{\partial}u=f$ satisfies
$$
 \int_{\Omega}|u(z)|^p(1+|z|^2)^{p}dA(z)\le C_p\int_{\Omega}|f(z)|^p dA(z).
$$
In particular, this is coherent in the power of $(1+|z|^2)$ with H\"ormander's result which
yields exactly the above estimate for $p=2$ (with a simpler proof). However, our argument 
does not work for $p=2$.\\

The proof below follows essentially the argument \cite[Theorem 9.4.1]{AM} combined with the $L^p$-estimates from the previous lemma.

\begin{proof}[Proof of Theorem \ref{thmCousin}]
%

Pick $f \in A^p(\Omega_1 \cap \Omega_2)$. 
Take $\chi$ a bounded $C^{\infty}$-function such that $\chi=1$ on ${\Omega_1 \setminus \Omega_2}$ and $\chi=0$ on $\Omega_2 \setminus \Omega_1$. Since $\dist(\Omega_1 \setminus \Omega_2,\Omega_2 \setminus \Omega_1)>0$, we can assume that $\nabla \chi$ is uniformly bounded (note that it vanishes outside $\Omega_1\cap\Omega_2$). So we can define $h_1=f(1-\chi)$ on $\Omega_1$ and $h_2=f \chi$ on $\Omega_2$. Using the analyticity of $f$, we have $\overline{\partial}h_1= -f\overline{\partial}\chi=-\overline{\partial}h_2 $ on $\Omega_1 \cap \Omega_2$, which implies the existence of a $C^\infty$-continuation $v$ such that $v=\overline{\partial}h_1$ on $\Omega_1$ and $v=-\overline{\partial}h_2$ on $\Omega_2$. Since $\overline{\partial}\chi$ is bounded, $v$ belongs to $L^p(\Omega_1 \cup \Omega_2)$ and by Lemma \ref{Hormander}, there exists $u \in L^p(\Omega_1 \cup \Omega_2, \omega_{1})$ such that $\overline{\partial} u = v$. Finally, defining $f_1=h_1-u$ on $\Omega_1$ and $f_2=h_2+u$ on $\Omega_2$, we obtain $f= f_1+f_2$ on $\Omega_1 \cap \Omega_2$ and $f_i \in A^p(\Omega_i, \omega_{1})$ by definition of $u$. The proof is complete.
\end{proof}

Now, a multiplier argument gives us a general separation of singularities result. 
\begin{proof}[Proof of Corollary \ref{corCousin}]
Let $z_0$ be such that $\dist(z_0, {\Omega_1 \cup \Omega_2})>0$,  and write $P(z)=(z-z_0)^2$. Pick $f \in A^p(\Omega_1 \cap \Omega_2)$, then $g:=Pf$ belongs also to $A^p(\Omega_1 \cap \Omega_2)$ since $P$ is bounded on $\Omega_1 \cap \Omega_2$. So, by Theorem \ref{thmCousin}, $g=g_1 + g_2$ with $g_1 \in A^p(\Omega_1, \omega_1)$ and $g_2 \in A^p(\Omega_2, \omega_1)$. Therefore $f=\frac{g}{P}$ belongs to $A^p(\Omega_1)+A^p(\Omega_2)$.  
\end{proof}

With the same argument used in Corollary \ref{weighted}, we obtain the following weighted version of Theorem \ref{thmCousin}.

\begin{cor}
\label{Cousinweighted}
Let $\Omega_1$ and $\Omega_2$ be open sets of $\C$ such that $\Omega_1 \cap \Omega_2 \neq \emptyset$. If $\dist(\Omega_1 \setminus \Omega_2,\  \Omega_2 \setminus \Omega_1) >0$, then for any $l\in \N$, we have $A^p(\Omega_1 \cap \Omega_2, \omega_l) \subset A^p(\Omega_1, \omega_{l+1}) + A^p(\Omega_2,\omega_{l+1})$.
\end{cor}

Before proving Theorem \ref{thm3} we need an auxiliary result on decompositions of Bergman spaces on unbounded intersections of half planes.
This is provided by the following lemma. 

\begin{lem}
\label{lemmaunbounded}
Let $n \geq 2$ and $H_1, \dots, H_n$ be half planes such that $\Omega:=\bigcap_{k=1}^n H_k$ is non-empty, convex and unbounded. Then for any $l \in \N$, we have $$A^p(\Omega, \omega_l) \subset \sum_{k=1}^n A^p(H_k, \omega_{l+(n-1)}).$$ 
\end{lem}
\begin{proof}
Pick $f \in A^p(\Omega, \omega_l)$. We can assume that $\Omega$ has exactly $n$ sides, 
otherwise just add the zero function. 
Let us prove the result by induction. If $n=2$, $\Omega$ is either a sector or a strip.
In the first case, the result follows from Corollary \ref{weighted}. When $\Omega$ is a strip,  then $d(H_1 \setminus \Omega, H_2\setminus \Omega) >0$ and the result follows from Corollary \ref{Cousinweighted}. So the base case of the induction is established.

Assume now that the lemma is true for every $2\leq n < N$. We shall prove it for $N$. Since $\Omega$ is convex and $N\geq 3$, the boundary $\partial \Omega$ is path connected. Without loss of generality we can assume that the half planes $H_1, \dots H_N$ are ordered such that the sides $S_1, \dots, S_N$ of $\partial \Omega$ satisfy $S_i \subset \partial H_i$, $S_i\cap S_{i+1}\neq\emptyset$ ($1 \leq i \leq N-1$), and $S_1$ and $S_N$ are the unbounded sides of $\partial \Omega$. 
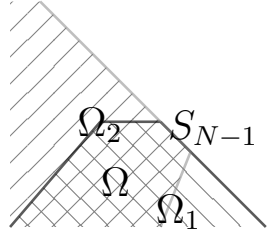
\begin{figure}[h]
\begin{center}
\begin{tikzpicture}[scale=0.2]

\fill[line space=10pt, pattern=my north west lines, pattern color=gray](-7,-7)-- (-1,0) -- (3,0) -- ( 10, -7)-- cycle ;
\fill[line space=10pt, pattern=my north east lines, pattern color=gray] (-7,-7)--(-7,8)--(-5,8)-- (5,-2) -- (-7,0) -- (5,-2)-- (3,-7)-- cycle ;

\draw  [line width=1.pt, color=black,  opacity=0.65](-7,-7)-- (-1,0);
\draw  [line width=1.pt, color=black, opacity=0.65](-1,0)-- (3,0);
\draw  [line width=1.pt, color=black, opacity=0.65](3,0)-- (10,-7);
\draw  [line width=1.pt, color=gray, opacity=0.5](5,-2)-- (3,-7);
\draw  [line width=1.pt, color=gray, opacity=0.5](5,-2)-- (-5,8);

\draw[color=black] (6.5,-0.5) node[xscale=1.3, yscale=1.3] {$S_{N-1}$};
\draw[color=black] (0,-4) node[xscale=1.3, yscale=1.3] {$\Omega$};
\draw[color=black] (4.2,-6) node[xscale=1.3, yscale=1.3] {$\Omega_1$};
\draw[color=black] (-1,-0.2) node[xscale=1.3, yscale=1.3] {$\Omega_2$};
\end{tikzpicture}
\caption{\label{figure2} The decomposition of $\Omega$ as the intersection of $\Omega_1$ and $\Omega_2$.}
\end{center}
\end{figure}
Then, writing $\Omega=\left(\bigcap_{i=1}^{N-1} H_i \right) \cap \left(H_{N-1} \cap H_N \right):= \Omega_1 \cap \Omega_2$ (see Figure \ref{figure2}), we obtain $f=f_1 + f_2$ with $f_1 \in A^p(\Omega_1, \omega_{l+1})$ and $f_2 \in A^p(\Omega_2, \omega_{l+1})$. 
Indeed, this follows from Corollary \ref{Cousinweighted} since $\Omega$ is convex and unbounded, and $\dist(\Omega_1 \setminus \Omega_2,\  \Omega_2 \setminus \Omega_1)=|S_{N-1}| > 0$. 
Finally, using the induction hypothesis on $\Omega_1$ 
and $\Omega_2$ (observe that $\Omega_2$ is a sector so that the weight is not changed here),
we conclude the inductive step, which proves the lemma.  
\end{proof}

\begin{proof}[Proof of Theorem \ref{thm3}]
Again we will prove the result by induction. 
If $n=3$, $\mathcal{P}$ is a triangle and we make the decomposition $\mathcal{P}=\Sigma_3 \cap \Omega$ where $\Sigma_3$ is the angular sector $H_1 \cap H_2$ and $\Omega$ is an unbounded domain, as in Figure \ref{figureTriangle}. We will again denote by $S_i$ the sides of $\mathcal{P}$ and $S_i\subset\partial H_i$. As previously, $\dist(\Sigma_3 \setminus \Omega, \ \Omega \setminus \Sigma_3) >0$, so by Corollary \ref{corCousin} we have 
\begin{eqnarray}\label{Incl1}
 A^p(\mathcal{P})=A^p(\Sigma_3) + A^p(\Omega). 
\end{eqnarray}
(Observe that the corollary does not require any convexity assumption, and no weights appear here.)
Hence, it remains to decompose $\Omega$ which will be done writing $\Omega= \Theta_1 \cap \Theta_2$, as in Figure \ref{figureTriangle}. Again, $\dist(\Theta_1 \setminus \Omega, \ \Theta_2 \setminus \Omega) >0$, so by Theorem \ref{thm2}, we have 
\begin{eqnarray}\label{Incl2}
A^p(\Omega)\subset A^p(\Theta_1, \omega_1) + A^p(\Theta_2, \omega_1).
\end{eqnarray} 

Define $\Sigma_2=H_1\cap H_3$ and $\Sigma_1=H_2\cap H_3$. Clearly $\Sigma_k\subset \Theta_k$, $k=1,2$, and so 
\begin{eqnarray}\label{Incl3}
 A^p(\Theta_k,\omega_1)\subset A^p(\Sigma_k,\omega_1), \quad k=1,2. 
\end{eqnarray}
From \eqref{Incl1}-\eqref{Incl3}, we obtain $A^p(\mathcal{P}) \subset \sum_{k=1}^3 A^p(\Sigma_k, \omega_1)$. It can be checked that $\overline{\bigcup_{k=1}^3 \Sigma_k}\neq\C$, so that there exists $z_0$ with $\dist(z_0,\overline{\bigcup_{k=1}^3 \Sigma_k})>0$. Define $P(z)=(z-z_0)^2$.
Then, for $f \in A^p(\mathcal{P})$, 
the function $g=Pf$ is also in $A^p(\mathcal{P})$ (multiplication by $P$ is actually an invertible operation on $A^p(\mathcal{P})$), which implies that $g$ can be written as $g=g_1+g_2+g_3$ with $g_i \in A^p(\Sigma_i, \omega_1)$. Thus, $f =\frac{g}{P}=\sum_{k=1}^3\frac{g_k}{P}$.  Hence, setting $f_k=\frac{g_k}{P}$, and since $\|g_k\|_{A^p(\Sigma_k,\omega_1)}$ is comparable to $\|f_k\|_{A^p(\Sigma_k)}=\|g_k/P\|_{A^p(\Sigma_k)}$, we obtain that $f$ belongs to  $\sum_{k=1}^3 A^p(\Sigma_k)$. This means that $A^p(\mathcal{P}) = \sum_{k=1}^3 A^p(\Sigma_k) $ and using Theorem \ref{thm2}, we obtain the base case of the induction.

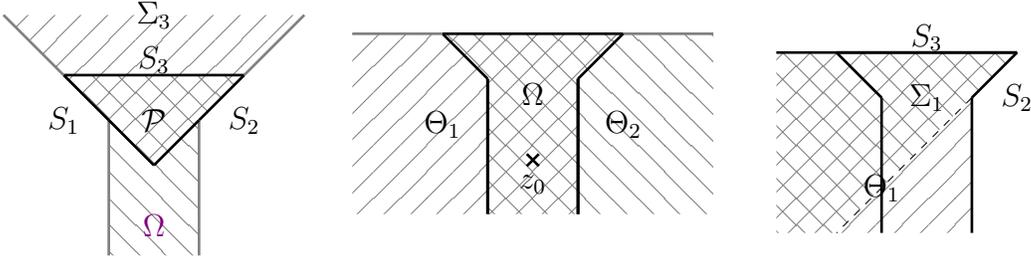
\begin{figure}[h]
\begin{center}
\begin{minipage}[c]{4.5cm}
\begin{tikzpicture}[scale=0.4]

\fill[line space=10pt, pattern=my north west lines, pattern color=gray] (-1.5,-3)-- (-1.5,1.5) -- (-3,3) -- (3,3) -- (1.5,1.5)--(1.5,-3) --  cycle ;
\fill[line space=10pt, pattern=my north east lines, pattern color=gray] (-5,5) -- (5,5) -- (0,0) -- cycle ;

\draw  [line width=1.pt, color=gray](-1.5,-3)-- (-1.5,1.5);
\draw  [line width=1.pt, color=black](-1.5,1.5) -- (-3,3);
\draw  [line width=1.pt, color=black](-3,3) -- (3,3);
\draw  [line width=1.pt, color=black](3,3) -- (1.5,1.5);
\draw  [line width=1.pt, color=gray](1.5,1.5)--(1.5,-3);

\draw  [line width=1.pt, color=gray](-5,5) -- (-3,3);
\draw  [line width=1.pt, color=black](-1.5,1.5) -- (0,0);
\draw  [line width=1.pt, color=gray](5,5) -- (3,3);
\draw  [line width=1.pt, color=black](1.5,1.5) -- (0,0);

\draw[color=black] (0,1.5) node[xscale=1, yscale=1] {$\mathcal{P}$};
\draw[color=black] (0,5) node[xscale=1, yscale=1] {$\Sigma_3$};
\draw[color=violet] (0,-2) node[xscale=1, yscale=1] {$\Omega$};
\draw[color=black] (-3,1.5) node[xscale=1, yscale=1] {$S_1$};
\draw[color=black] (3,1.5) node[xscale=1, yscale=1] {$S_2$};
\draw[color=black] (0,3.5) node[xscale=1, yscale=1] {$S_3$};
\end{tikzpicture}
\end{minipage}
\begin{minipage}[c]{5.5cm}
\begin{tikzpicture}[scale=0.4]

\fill[line space=10pt, pattern=my north west lines, pattern color=gray] (-1.5,-3)-- (-1.5,1.5) -- (-3,3) -- (6,3) --(6,-3) -- cycle ;
\fill[line space=10pt, pattern=my north east lines, pattern color=gray] (-6,-3)--(-6,3)-- (3,3) -- (1.5,1.5)--(1.5,-3) -- cycle ;

\draw  [line width=1.pt](-1.5,1.5) -- (-3,3);
\draw  [line width=1.2pt](-1.5,-3)-- (-1.5,1.5);
\draw  [line width=1.pt, color=gray](3,3) -- (6,3);
\draw  [line width=1.pt, color=gray] (-6,3) -- (-3,3) ;
\draw  [line width=1.pt] (-3,3) -- (3,3) ;
\draw  [line width=1.pt](3,3) -- (1.5,1.5);
\draw  [line width=1.pt](1.5,1.5)--(1.5,-3);

\draw  [line width=1.pt](1.5,1.5)--(1.5,-3);
\draw  [line width=1.pt](1.5,1.5)--(1.5,-3);

\draw  [line width=1.pt](-0.2,-1)--(0.2,-1.4);
\draw  [line width=1.pt](-0.2,-1.4)--(0.2,-1);

\draw[color=black] (3,0) node[xscale=1, yscale=1] {$\Theta_2$};
\draw[color=black] (-3,0) node[xscale=1, yscale=1] {$\Theta_1$};
\draw[color=black] (0,1) node[xscale=1, yscale=1] {$\Omega$};
\draw[color=black] (0,-2) node[xscale=1, yscale=1] {$z_0$};
\end{tikzpicture}
\end{minipage}
\begin{minipage}[c]{4cm}
\begin{tikzpicture}[scale=0.4]

\fill[line space=10pt, pattern=my north west lines, pattern color=gray] (-5,-3)--(-5,3)-- (3,3) -- (-3,-3)--(-5,-3) -- cycle ;
\fill[line space=10pt, pattern=my north east lines, pattern color=gray] (-5,-3)--(-5,3)-- (3,3) -- (1.5,1.5)--(1.5,-3) -- cycle ;

\draw  [line width=1.pt](-1.5,-3)-- (-1.5,1.5);
\draw  [line width=1.pt](-1.5,1.5) -- (-3,3);
\draw  [line width=1.pt](-3,3) -- (3,3);
\draw  [line width=1.pt] (-5,3) -- (3,3) ;
\draw  [line width=1.pt](3,3) -- (1.5,1.5);
\draw  [line width=1.pt](1.5,1.5)--(1.5,-3);
\draw [dashed](1.5,1.5)--( -3,-3);

\draw[color=black] (-1.5,-1.5) node[xscale=1, yscale=1] {$\Theta_1$};
\draw[color=black] (0,1.5) node[xscale=1, yscale=1] {$\Sigma_1$};
\draw[color=black] (3,1.5) node[xscale=1, yscale=1] {$S_2$};
\draw[color=black] (0,3.5) node[xscale=1, yscale=1] {$S_3$};

\end{tikzpicture}
\end{minipage}

\caption{\label{figureTriangle} Decomposition of a triangle.}
\end{center}
\end{figure}

Now, assume that the result is true for every $3\leq k < n$. There are two cases. First suppose that $\mathcal{P}$ has two non-consecutive non-parallel sides (see Figure \ref{figureinduction}). Denote by $S_1, \dots, S_n$ its ordered sides ($S_i\cap S_{i+1}\neq\emptyset$, $S_{n}\cap S_1\neq\emptyset$) and $H_1, \dots, H_n$ the corresponding half planes.  Let $S_i$ and $S_j$ be two non-parallel sides with $j \geq i+2$. Write $\Omega_1= \bigcap_{k =i}^j H_k$ and $\Omega_2 = \bigcap_{k \notin (i, j)} H_k$ (we denote $[i,j]=\{i,i+1,\ldots,j\}$ and $(i,j)=\{i+1,\ldots,j-1\}$). Observe that $\Omega_1$ and $\Omega_2$ have the sides $S_i$ and $S_j$ in common. It is clear that $\mathcal{P}= \Omega_1 \cap \Omega_2$ and $\dist(\Omega_1 \setminus \Omega_2, \, \Omega_2 \setminus \Omega_1) = \min\limits_{\substack{l \in (i, j) \\ k \notin [i, \, j]}} \dist(S_l, S_k) >0$. So, by Corollary \ref{corCousin}, we have $A^p(\mathcal{P})= A^p(\Omega_1) + A^p(\Omega_2)$. Since $S_i$ and $S_j$ are non-parallel, one of the sets $\Omega_i$ is a polygon and the other one is unbounded. Les us assume that $\Omega_1$ is a polygon. Using Lemma \ref{lemmaunbounded} with $l=0$, we obtain $A^p(\Omega_2) \subset  \sum_{k \notin (i,\, j)} A^p(H_k, \omega_{n-1})$, and hence
$$
A^p(\mathcal{P})\subset A^p(\Omega_1) +\sum_{k \notin (i,\, j)} A^p(H_k, \omega_{n-1}),
$$ 
where $\Omega_1$ is a polygon with lower degree. By the induction hypothesis, $ A^p(\Omega_1)$-functions decompose in the required way into $A^p(H_k)$-functions where $k\in [i,j]$. 

In order to manage the second term, we need to get rid of the weight $\omega_{n-1}$. This will again be done using the multiplication by an appropriate polynomial vanishing neither on $H_k$, $k\notin (i,j)$ nor on $\Omega_1$.
For that, observe that $\Omega_2$ is an unbounded convex domain, and $S_i$ and $S_j$ are non-parallel. We claim that there exists $z_0 \notin \overline{\bigcup_{k \notin (i, \, j)} H_k \cup \Omega_1}$. Indeed, for every $k\notin (i,j)$, $H_k\subset H_i\cup H_j$, and moroever $\Omega_1 \subset H_i\cup H_j$ (convexity comes into play here). Since $S_i$ and $S_j$ are non parallel, $H_i\cup H_j\neq \C$, and it is enough to pick $z_0\notin H_i\cup H_j$ with $\dist(z_0,H_i\cup H_j)>0$. 
The polynomial we are looking for is $P(z)=(z-z_0)^{2(n-1)}$. Pick now $f \in A^p(\mathcal{P})$. As in the previous corollary, writing $g=P f$, we have $g \in A^p(\mathcal{P})$ and so, by the reasoning above, $g=g_1+\sum_{k \notin (i, \, j)} g_{2,k}$ with $g_1 \in A^p(\Omega_1)$ and 
$g_{2, k} \in A^p(H_k, \omega_{n-1})$, and hence $f_{2,k}=g_{2,k}/P\in A^p(H_k)$. Also, since multiplication (and division) by $P$ is an invertible operation on $A^p(\Omega_1)$, we have $f_1=g_1/P\in A^p(\Omega_1)$, and by the induction hypothesis (applied to $\Omega_1$) $f_1$ decomposes into a sum of $A^p(H_k)$-functions, $k\in [i,j]$:
$$
 f=f_1+\sum_{k \notin (i, \, j)} \frac{g_{2,k}}{P}\in \sum_{k\in [i,j]}A^p(H_k)+\sum_{k\notin(i,j)}
 A^p(H_k).
$$

\begin{figure}[h]
\begin{center}
\begin{minipage}[c]{6cm}
\begin{tikzpicture}[scale=0.5]

\fill[line space=10pt, pattern=my north west lines, pattern color=gray] (-0.3, 1.10551724138)  -- (1.16,4.68) -- (2.74,4.9)  -- (6.765520123323564,1.41121589312)  -- cycle;
\fill[line space=10pt, pattern=my north east lines, pattern color=gray] (0.58,3.26)-- (0.9,2.62) -- (2.46,2.74)-- (3.64,4.12) -- (2.74,4.9)-- (1.6394452149791956,  5.85381414702) -- (1.16,4.68) -- cycle;

\draw [line width=1.pt] (0.58,3.26)-- (0.9,2.62);
\draw [line width=1.pt] (0.9,2.62)-- (2.46,2.74);
\draw [line width=1.pt] (2.46,2.74)-- (3.64,4.12);
\draw [line width=1.pt] (3.64,4.12)-- (2.74,4.9);
\draw [line width=1.pt ] (2.74,4.9)-- (1.16,4.68);
\draw [line width=1.pt] (1.16,4.68)-- (0.58,3.26);
\draw [line width=1.pt,domain=-0.3:1.6394452149791956] plot(\x,{(-0.8821791955617195-1.1738141470180317*\x)/-0.4794452149791957});
\draw [line width=1.pt,domain=1.6394452149791956 : 6.765520123323564] plot(\x,{(--14.553369209431349-1.7338141470180313*\x)/2.0005547850208045});

\draw[color=black] (2,4) node[xscale=1, yscale=1] {$\mathcal{P}$};
\draw[color=black] (1.8,5.2) node[xscale=1, yscale=1] {$\Omega_1$};
\draw[color=black] (3,2) node[xscale=1, yscale=1] {$\Omega_2$};
\draw[color=black] (3.5,4.7) node[xscale=1, yscale=1] {$S_2$};
\draw[color=black] (3.3,3.2) node[xscale=1, yscale=1] {$S_3$};
\draw[color=black] (2,2.5) node[xscale=1, yscale=1] {$\ldots$};
\draw[color=black] (0.5,4.2) node[xscale=1, yscale=1] {$S_6$};
\end{tikzpicture}
\end{minipage}
\begin{minipage}[c]{5cm}
\begin{tikzpicture}[scale=0.5,grid/.style={pattern=MyGrid}]

\fill[line space=10pt, pattern=my north west lines, pattern color=gray] (-0.3, 1.10551724138)  -- (2.5,8) -- (6.765,8) -- (6.765520123323564,1.41121589312)  -- cycle;
\fill[line space=10pt, pattern=my north east lines, pattern color=gray] (-0.3, 1.10551724138)-- (-0.3,7.5 )--(6.765520123323564,1.41121589312)  -- cycle;

\draw [line width=1.pt] (0.58,3.26)-- (0.9,2.62);
\draw [line width=1.pt] (0.9,2.62)-- (2.46,2.74);
\draw [line width=1.pt] (2.46,2.74)-- (3.64,4.12);
\draw [line width=1.pt ] (2.74,4.9)-- (1.16,4.68);

\draw [line width=1.pt] (1.4,7.1)-- (1.6,7.3);
\draw [line width=1.pt] (1.4,7.3)-- (1.6,7.1);

\draw [dashed] (-0.3,4.47)--(6.7,5.459);
\draw [line width=1.pt,domain=-0.3:2.5] plot(\x,{(-0.8821791955617195-1.1738141470180317*\x)/-0.4794452149791957});
\draw [line width=1.pt,domain=-0.3 : 6.765520123323564] plot(\x,{(--14.553369209431349-1.7338141470180313*\x)/2.0005547850208045});

\draw[color=black] (2,4) node[xscale=1, yscale=1] {$\mathcal{P}$};
\draw[color=black] (1.8,5.2) node[xscale=1, yscale=1] {$\Omega_1$};
\draw[color=black] (3,2) node[xscale=1, yscale=1] {$\Omega_2$};
\draw[color=black] (0.5,6) node[xscale=1, yscale=1] {$H_2$};
\draw[color=black] (3.5,6) node[xscale=1, yscale=1] {$H_6$};
\draw[color=black] (5,4.5) node[xscale=1, yscale=1] {$H_1$};
\draw[color=black] (1.5,7.5) node[xscale=1, yscale=1] {$z_0$};
\end{tikzpicture}
\end{minipage}
\caption{\label{figureinduction} Decomposition of $\mathcal{P}$ as $\Omega_1 \cap \Omega_2$ with a possible numbering of $S_i$, and the domain $H_i\cup H_j$. ($S_2$ and $S_6$ are two non-consecutive non-parallel sides.)}
\end{center}
\end{figure}
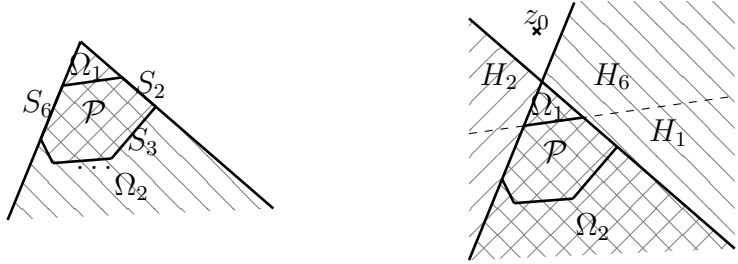

Secondly, suppose that {\it all} non-consecutive sides of $\mathcal{P}$ are parallel. Then $n=4$ and $\mathcal{P}$ is a parallelogram. We treat this case directly. As in the first case, we denote by $S_1, \dots, S_4$ the consecutive sides of $\mathcal{P}$,  $H_1, \dots, H_4$ the corresponding half planes and $\Sigma_k=H_k \cap H_{k+1}$ (with $H_5=H_1$) the angular sectors. We make the same decomposition : $A^p(\mathcal{P}) = A^p(\Omega_1) + A^p(\Omega_2)$ where $\Omega_1 = H_1\cap H_2 \cap H_3$ and $\Omega_2 = H_3 \cap H_4 \cap H_1$. This time $\Omega_1$ and $\Omega_2$ are both unbounded and $\dist(\Omega_1\setminus\Omega_2,\Omega_2\setminus\Omega_1)>0$. Using Theorem \ref{thmCousin} we get $A^p(\mathcal{P})\subset A^p(\Omega_1,\omega_1)+A^p(\Omega_2,\omega_1)$. Next we apply
Lemma \ref{lemmaunbounded} to each of these spaces to get $A^p(\mathcal{P})\subset \sum_{k=1}^4A^p(H_k,\omega_3)$ (notice that the weight is given by $1+(n-1)=3$ since we intersect 3 half planes). Observe that in this case $H_1\cup H_2\cup H_3=\C$ (and similarly for $H_3$, $H_4$, $H_1$), so that at this step we cannot find a $z_0$ allowing the multiplication trick by a polynomial.  Instead, we use Corollary \ref{weighted}, to get sectors: 
$$
A^p(\mathcal{P}) \subset \sum_{k=1}^4A^p(H_k,\omega_3)=A^p(\Sigma_1, \omega_3) + A^p(\Sigma_3, \omega_3),
$$
where $\Sigma_1=H_1\cap H_2$ and $\Sigma_3=H_3\cap H_4$.
Now, there exists a point $z_0$ which is not in $\overline{\Sigma_1 \cup \Sigma_3}$, so that we can use the multiplication trick as in the first case: $f\in A^p(\mathcal{P})$ implies $g=fP\in A^p(\mathcal{P})$ splits into $g=g_1+g_2$ with $g_1\in A^p(\Sigma_1,\omega_3)$, $g_2\in A^p(\Sigma_3,\omega_3)$. With the suitable choice of $P$ we compensate again the weight so that $f_1=g_1/P\in A^p(\Sigma_1)$ and $f_2=g_2/P\in A^p(\Sigma_3)$. Hence $A^p(\mathcal{P}) \subset A^p(\Sigma_1) + A^p(\Sigma_3)$, and we conclude using Theorem \ref{thm2}.
\end{proof}

Let us indicate how the above results apply to more general separation of singularities problem in Bergman spaces, and not only on polygons.

\begin{proof}[Proof of Theorem \ref{thmcvx}]
Assume that $\Omega_1 \nsubset \Omega_2$ and $\Omega_2 \nsubset \Omega_1$, otherwise the problem is trivial. 
Let us make some additional observations. 
\begin{itemize}
\item By assumption, $\Omega_1$ and $\Omega_2$ are convex, so that $\Omega_1 \cap \Omega_2$ is also convex (and non-empty), hence $\partial (\Omega_1 \cap \Omega_2)$ is the image $\Gamma$ of a Jordan curve $\gamma$.
\item $\partial \Omega_1 \cap \partial \Omega_2$ includes at least two points 
and if we write $\tilde{\partial} \left(\partial \Omega_1 \cap \partial \Omega_2 \right) :=\{z_1, \dots, z_n\}$ $(n \geq 2)$, we have 
$$ 
\partial(\Omega_1\cap\Omega_2)=\Gamma=\bigcup_{k=1}^n \Gamma_{z_k, z_{k+1}}, \quad (z_{n+1}:=z_1)
$$
where $\Gamma_{z_k, z_{k+1}}$ is the path in $\partial (\Omega_1\cap \Omega_2)$ connecting $z_k$ to $z_{k+1}$. Moreover, $\Gamma_{z_k, z_{k+1}} \subset \partial \Omega_1$ or $\Gamma_{z_k, z_{k+1}} \subset \partial \Omega_2$. Note that it can happen that $\Gamma_{z_k,z_{k+1}}\subset \partial\Omega_1\cap
\partial\Omega_2$.
\item By convexity, $[z_k,z_{k+1}]\subset \overline{\Omega_1\cap\Omega_2}$.
\end{itemize}
Let us start assuming $n \geq 3$. Pick $f \in A^p(\Omega_1 \cap \Omega_2)$. Write 
$$\mathcal{P}_{z_1, \dots, z_n}:=\mathrm{Int}\left(\mathrm{Conv}(\{z_1, \dots, z_n\})\right)=\bigcap_{k=1}^n H_{z_k, z_{k+1}}
$$ 
where $H_{z_k, z_{k+1}}$ is the half plane associated to the side $[z_k, z_{k+1}]$ of the polygon 
(see Figure \ref{figconvex}), and note that by convexity $\mathcal{P}_{z_1, \dots, z_n} \subset \Omega_1 \cap \Omega_2$. So $f \in A^p(\mathcal{P}_{z_1, \dots, z_n})$ and by Theorem \ref{thm3} there exist $f_1, \dots, f_n$ such that $f_k \in A^p(H_{z_k, z_{k+1}})$ and $f=\sum_{k=1}^n f_k$. It remains to prove that each $f_k$ belongs either to $A^p(\Omega_1)$ or to $A^p(\Omega_2)$. Let $i_k \in\{1,\, 2\}$ be the index such that $\Gamma_{z_k, z_{k+1}} \subset \partial \Omega_{i_k}$ (when $\Gamma_{z_k,z_{k+1}}$ is in both boundaries, we can pick either of the values $1$ or $2$ for $i_k$). Two cases may occur. If $\Gamma_{z_k, z_{k+1}} = [z_k, \, z_{k+1}]$, then, by convexity, $\Omega_{i_k} \subset H_{z_k, z_{k+1}}$ and the result follows: $f_k\in A^p(H_{z_k,z_{k+1}})\subset A^p(\Omega_{i_k})$. 
So, assume that $\Gamma_{z_k, z_{k+1}} \neq [z_k, \, z_{k+1}]$ and write 
$$
O_{z_k, \, z_{k+1}}:=\bigcap_{j\neq k}H_{z_j,z_{j+1}}\cap \Omega_{i_k} \quad \subset \quad \Omega_1 \cap \Omega_2
$$ 
(one side of the polygon has been replaced by the arc $\Gamma_{z_k,z_{k+1}}$).
Since $f_k\in A^p(H_{z_k,z_{k+1}})$ it is obvioulsy in $A^p(H_{z_k,z_{k+1}}\cap \Omega_{i_k})$. We claim that it extends to a function in $A^p(\Omega_{i_k})$. 
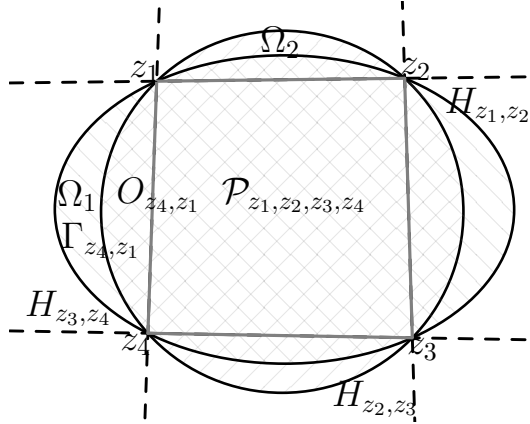
\begin{figure}[h]
\begin{center}
\begin{tikzpicture}[scale=0.6, line cap=round,line join=round,>=triangle 45,x=1.0cm,y=1.0cm]
\clip(-4.3,-3.02) rectangle (7.28,6.3);

\draw [line width=1.pt, line space=10pt, pattern=my north east lines, fill opacity =0.2] (1.74,1.64) circle (4.0144239935512545cm);
\draw [rotate around={-0.15197783856389324:(1.79,1.69)},line width=1.pt, line space=10pt, pattern=my north west lines, fill opacity = 0.2] (1.79,1.69) ellipse (5.089130133829856cm and 3.418515104406456cm);

\draw [line width=1.1pt,domain=-4.3:7.28, dash pattern=on 5pt off 5pt] plot(\x,{(-6.703381318860409-5.588763184581731*\x)/-0.19473437581419883});
\draw [line width=1.1pt,domain=-4.3:7.28, dash pattern=on 5pt off 5pt] plot(\x,{(--26.39968235467491-5.753876824021156*\x)/0.16993140901817227});
\draw [line width=1.1pt,domain=-4.3:7.28, dash pattern=on 5pt off 5pt] plot(\x,{(--24.979994797982414--0.06486131856383803*\x)/5.493748136354552});
\draw [line width=1.1pt,domain=-4.3:7.28, dash pattern=on 5pt off 5pt] plot(\x,{(-6.299124198416687-0.10025232087558678*\x)/5.858413921186923});

\draw [line width=1.2pt, color= gray] (-1.0414327914516701,4.534690247096806)-- (4.452315344902882,4.599551565660644);
\draw [line width=1.2pt, color= gray] (4.452315344902882,4.599551565660644)-- (4.622246753921054,-1.1543252583605128);
\draw [line width=1.2pt, color= gray] (4.622246753921054,-1.1543252583605128)-- (-1.236167167265869,-1.054072937484926);
\draw [line width=1.2pt, color= gray] (-1.236167167265869,-1.054072937484926)-- (-1.0414327914516701,4.534690247096806);

\begin{scriptsize}
\draw[color=black] (-1.3,4.75) node {\fontsize{14pt}{16pt}$z_1$};
\draw[color=black] (4.7,4.8) node {\fontsize{14pt}{16pt}$z_2$};
\draw[color=black] (4.85,-1.4) node {\fontsize{14pt}{16pt}$z_3$};
\draw[color=black] (-1.5,-1.3) node {\fontsize{14pt}{16pt}$z_4$};
\draw[color=black] (-2.8,2) node {\fontsize{14pt}{16pt}$\Omega_1$};
\draw[color=black] (1.7,5.4) node {\fontsize{14pt}{16pt}$\Omega_2$};
\draw[color=black] (2,2) node {\fontsize{14pt}{16pt}$\mathcal{P}_{z_1,z_2,z_3,z_4}$};
\draw[color=black] (-1,2) node {\fontsize{14pt}{16pt}$O_{z_4,z_1}$};
\draw[color=black] (-2.28,1) node {\fontsize{14pt}{16pt}$\Gamma_{z_4,z_1}$};
\draw[color=black] (3.8,-2.5) node {\fontsize{14pt}{16pt}$H_{z_2,z_3}$};
\draw[color=black] (-3,-0.5) node {\fontsize{14pt}{16pt}$H_{z_3,z_4}$};
\draw[color=black] (6.3,4) node {\fontsize{14pt}{16pt}$H_{z_1,z_2}$};
\end{scriptsize}
\end{tikzpicture}
\caption{\label{figconvex} Intersection of two convex open sets.}
\end{center}
\end{figure}
By definition  $O_{z_k, \, z_{k+1}} \subset H_{z_j, \, z_{j+1}}$, $j\neq k$. 
Therefore, we have 
$$f_k=\underbrace{f}_{\in A^p(\Omega_1 \cap \Omega_2)\subset A^p(O_{z_k, \, z_{k+1}})} - \quad \sum_{j \neq k} \underbrace{f_j}_{\in A^p(H_{z_j,z_{j+1}}) \subset A^p(O_{z_k, \, z_{k+1}})} \in A^p(O_{z_k, \, z_{k+1}}). $$
Thus $f_k \in A^p(H_{z_k,z_{k+1}}\cup O_{z_k,z_{k+1}})\subset A^p(\Omega_{i_k})$. 

Finally, assume that $n=2$. It is sufficient to add a point $z_3$ which belongs to $\partial \left(\Omega_1 \cap \Omega_2\right) \setminus \{z_1, \, z_2\}$, and construct $\mathcal{P}_{z_1, \, z_2, \, z_3}$ as before. The rest of the proof is dealt with as in the previous case. 
\end{proof}

It it worth mentioning that convex sets with infinitely many intersections can be constructed easily. 
\\

\section{Reachable states of the heat equation \label{sec3}}

As already discussed in the first section, the result \cite[Theorem 1.1]{Or}, which states that  
\begin{equation}
\label{decomp}
\mathrm{Ran}\Phi_\tau = A^2(\Delta) + A^2(\pi-\Delta)
\end{equation} 
together with our Corollary 
\ref{coro1} yield the final characterization of the reachable states of the 1-D heat equation with $L^2$-boundary controls
\begin{equation}\label{reachS3}
 \mathrm{Ran}\Phi_\tau=A^2(D),
\end{equation}
as stated in Corollary \ref{reach}.

In this section we would like to make some additional observations on this and related control problems. Also, the general arguments presented in Section \ref{sec2} leading to Theorem \ref{thm3} might hide the very simple ideas which are actually behind Corollary \ref{coro1} and thus leading to \eqref{reachS3}. For this reason, we would like to present here a more direct proof of Corollary \ref{coro1} based on Corollary \ref{corCousin} and Theorem \ref{thm1} (case $p=2$).

\subsection{A direct proof to Corollary \ref{reach}}

As already mentioned several times (see for instance Remark \ref{rmk1}), the decomposition \eqref{decomp} is invariant by rotation and dilation. So, writing $\Sigma_2:=(1+i)-\C^{++}$ and denoting by $D'$ the square $D':=\C^{++} \cap \Sigma_2$ (see Figure \ref{figsquare}), it is enough to show that $A^2(D')= A^2(\C^{++}) + A^2(\Sigma_2)$. Let $z_0 \in \C \setminus \overline{(\C^{++} \cup \Sigma_2)}$ and $P(z)=(z-z_0)^2$ 
which is bounded and non-vanishing on $\overline{D'}$, so that multiplication by $P$ is an isomorphism onto $A^p(D')$.
In particular $f \in A^2(D')$ if and only if $g=P f\in A^2(D')$. 
\\

\begin{figure}[h]
\begin{center}
\begin{tikzpicture}[scale=0.3]


\fill[line width=2.pt,line space=10pt, pattern=my north east lines,fill] (0,8) -- (8,8) -- (8,0) -- (0,0) -- cycle ;
\fill[line width=2.pt,line space=10pt, pattern=my north west lines,fill] (-5,3) -- (3,3) -- (3,-5) -- (-5,-5) -- cycle ;

\draw  [line width=2.pt](0,8)-- (0,0);
\draw  [line width=2.pt](0,0)-- (8,0);
\draw  [line width=2.pt](-5,3)-- (3,3);
\draw  [line width=2.pt](3,3)-- (3,-5);

\draw  [line width=1.pt,->](-5,0)-- (8.3,0);
\draw  [line width=1.pt,->](0,-5)-- (0,8.3);

\draw[color=black] (1.5,1.5) node[xscale=1, yscale=1] {$D'$};
\draw[color=black] (6.5,7) node[xscale=1, yscale=1] {$\C^{++}$};
\draw[color=black] (-4,-4) node[xscale=1, yscale=1] {$\Sigma_2$};
\end{tikzpicture}
\caption{\label{figsquare} The decomposition of $D'$ as the intersection of $\C^{++}$ and $\Sigma_2$.}
\end{center}
\end{figure}
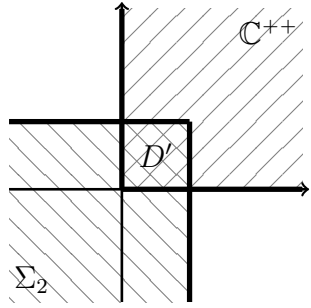

The proof decomposes into 3 steps. We have already met some arguments in the proof of Theorem \ref{thm3} when considering the case of a parallelogramme $\mathcal{P}$.
\\

\textbf{Step 1:} Let $S_1$ and $S_2$ be the half strips $S_1=\enstq{z=x+iy \in \C}{y >0, \ 0<x<1}$ and $S_2=(1+i)-S_1$. Note that $D'= S_1 \cap S_2$ (see Figure \ref{figure3}). Since $D'$ is bounded and $\dist(S_1 \setminus D', \, S_2 \setminus D') > 0$, by Corollary \ref{corCousin}, there exist $g_1 \in A^2(S_1)$ and $g_2 \in A^2(S_2)$ such that $g=g_1 + g_2$ on $D'$. This step is complete. 
\\


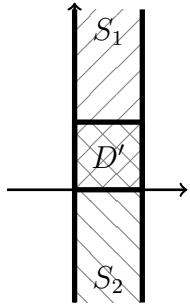
\begin{figure}[h]
\begin{center}
\begin{tikzpicture}[scale=0.3]

\fill[line width=2.pt,line space=10pt, pattern=my north east lines] (0,8) -- (3,8) -- (3,0) -- (0,0) -- cycle ;
\fill[line width=2.pt,line space=10pt, pattern=my north west lines, ] (0,-5) -- (0,3) -- (3,3) -- (3,-5) -- cycle ;

\draw  [line width=2.pt](0,0) -- (0,8);
\draw  [line width=2.pt](3,8) -- (3,0);
\draw  [line width=2.pt](3.1,0) -- (-0.1,0);

\draw  [line width=2.pt](0,-5) -- (0,3);
\draw  [line width=2.pt](0,3) -- (3,3);
\draw  [line width=2.pt](3,3) -- (3,-5);

\draw  [line width=1.pt,->](-3,0)-- (5,0);
\draw  [line width=1.pt,->](0,-5)-- (0,8.3);

\draw[color=black] (1.5,1.5) node[xscale=1, yscale=1] {$D'$};
\draw[color=black] (1.5,7) node[xscale=1, yscale=1] {$S_1$};
\draw[color=black] (1.5,-4) node[xscale=1, yscale=1] {$S_2$};
\end{tikzpicture}
\caption{\label{figure3} Step 1 --- Decomposition of $D'$ as the intersection of $S_1$ and $S_2$.}
\end{center}
\end{figure}

\textbf{Step 2:} Denote by $Q_{1,1}$ and $Q_{1,2}$ (respectively $Q_{2,1}=\Sigma_2$ and $Q_{2,2}$) the left and right quarter planes the intersection of which is $S_1$ (resp. $S_2$) (see Figure \ref{figure4}). We can repeat the same argument as in the previous step (applying Theorem \ref{thmCousin}) and obtain $g_1= g_{1,1} + g_{1,2}$ (resp. $g_2= g_{2,1} + g_{2,2}$) with $g_{1,i} \in A^2(Q_{1,i}, \omega_1)$ (resp. $g_{2,i} \in A^2(Q_{2,i}, \omega_1)$).   
\\

\begin{figure}[h]
\begin{center}
\begin{minipage}[c]{.5\linewidth}
\begin{tikzpicture}[scale=0.3]

\fill[line width=2.pt,line space=10pt, pattern=my north east lines] (0,8) -- (0,0) -- (8,0) -- (8,8) -- cycle ;
\fill[line width=2.pt,line space=10pt, pattern=my north west lines] (-5,0) -- (3,0) -- (3,8) -- (-5,8) -- cycle ;

\draw  [line width=1.pt](0,8)-- (0,0);
\draw  [line width=1.pt](0,0)-- (8,0);
\draw  [line width=1.pt](-5,0) -- (3,0);
\draw  [line width=1.pt](3,0) -- (3,8);

\draw  [line width=1.pt,->](-5,0)-- (8.3,0);
\draw  [line width=1.pt,->](0,-2)-- (0,8.3);

\draw[color=black] (-0.4,5) node[xscale=1, yscale=1] {$Q_{1,1}$};
\draw[color=black] (1.5,2) node[xscale=1, yscale=1] {$S_1$};
\draw[color=black] (3,5) node[xscale=1, yscale=1] {$Q_{1,2}$};
\end{tikzpicture}
\end{minipage}
\hfill
\begin{minipage}[c]{.4\linewidth}
\begin{tikzpicture}[scale=0.3]

\fill[line width=2.pt,line space=10pt, pattern=my north east lines] (0,-5) -- (0,3) -- (8,3) -- (8,-5) -- cycle ;
\fill[line width=2.pt,line space=10pt, pattern=my north west lines, ] (-5,3) -- (3,3) -- (3,-5) -- (-5,-5) -- cycle ;

\draw  [line width=1.pt](0,-5)-- (0,3);
\draw  [line width=1.pt](0,3)-- (8,3);
\draw  [line width=1.pt](-5,3) -- (3,3);
\draw  [line width=1.pt](3,3) -- (3,-5);

\draw  [line width=1.pt,->](-5,0)--(8.3,0) ;
\draw  [line width=1.pt,->](0,-5)--(0,5) ;

\draw[color=black] (-0.4,-2) node[xscale=1, yscale=1] {$Q_{2,1}$};
\draw[color=black] (1.5,0) node[xscale=1, yscale=1] {$S_2$};
\draw[color=black] (3,-2) node[xscale=1, yscale=1] {$Q_{2,2}$};
\end{tikzpicture}
\end{minipage}
\caption{\label{figure4} Step 2 --- The decompositions of $S_1$ and $S_2$.}
\end{center}
\end{figure}
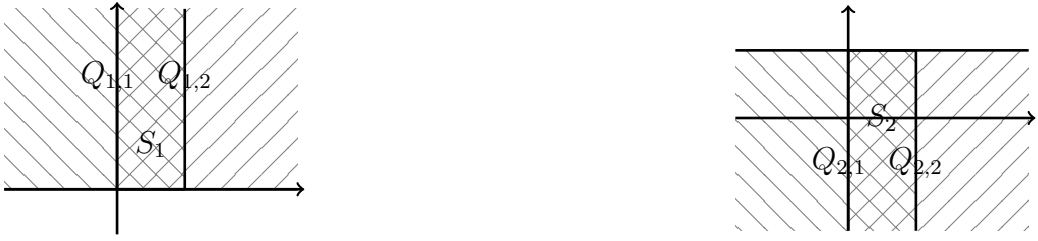

\textbf{Step 3:} Remark that $Q_{1,2}= \C^{++}$ and $Q_{2,1}=\Sigma_2$. So that we already have $g_{1,2}+g_{2,1} \in A^2(\C^{++}, \omega_1) + A^2(\Sigma_2, \omega_1)$, solving the problem for $g_{1,2}$ and $g_{2,1}$. Let us show the same for $g_{1,1}$ and $g_{2,2}$.  Denote by $H_{1,1,1}=\C^+$ the upper half plane by and $H_{1,1,2}=1-\C_+$  the left half plane translated by 1 (resp. $H_{2,2,1}$ and  the lower half plane translated by 1 and $H_{2,2,2}$ the right half plane) the intersection of which is $Q_{1,1}$ (resp. $Q_{2,2}$), see Figure \ref{figure5}. By Corollary \ref{weighted} with Remark \ref{rmk1}, $g_{1,1}$ belongs to 
$A^2(H_{1,1,1},\omega_1) + A^2(H_{1,1,2}, \omega_1)$ and $g_{2,2}$ belongs to 
$A^2(H_{2,2,1}, \omega_1) + A^2(H_{2,2,2}, \omega_1)$. 

Now, observing that $Q_{1,2}=\C^{++} \subset H_{1,1,1}$, $\C^{++}\subset H_{2,2,2}$, $\Sigma_2 \subset H_{1,1,2}$ and $\Sigma_2 \subset H_{2,2,1}$, we obtain that $g_{1,1}$ belongs to $A^2(\C^{++}, \omega_1)+A^2(\Sigma_2, \omega_1)$ and the same is true for $g_{2,2}$. Thus
$g\in A^2(\C^{++}, \omega_1) + A^2(\Sigma_2, \omega_1)$. Finally, by definition of $P$, $f=\frac{g}{P}$ belongs to $A^2(\C^{++}) + A^2(\Sigma_2)$, which concludes the proof. 
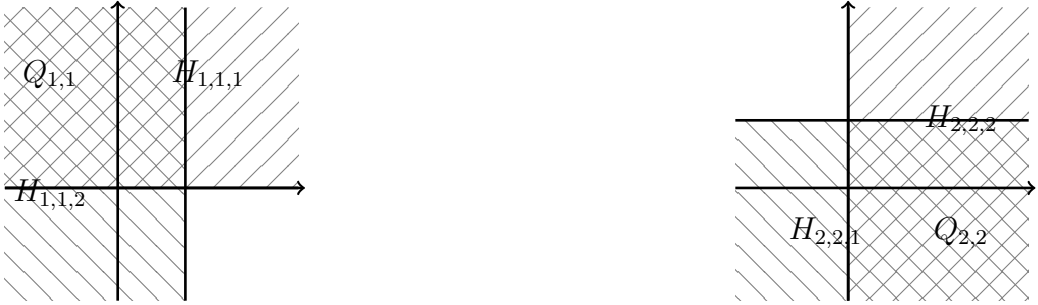
\begin{figure}[h]
\begin{center}
\begin{minipage}[c]{.5\linewidth}
\begin{tikzpicture}[scale=0.3]

\fill[line width=2.pt,line space=10pt, pattern=my north east lines] (-5,0) -- (8,0) -- (8,8) -- (-5,8) -- cycle ;
\fill[line width=2.pt,line space=10pt, pattern=my north west lines] (-5,-5) -- (3,-5) -- (3,8) -- (-5,8) -- cycle ;

\draw  [line width=1.pt](-5,0)-- (8,0);
\draw  [line width=1.pt](3,-5) -- (3,8);

\draw  [line width=1.pt,->](-5,0)-- (8.3,0);
\draw  [line width=1.pt,->](0,-5)-- (0,8.3);

\draw[color=black] (-3,5) node[xscale=1, yscale=1] {$Q_{1,1}$};
\draw[color=black] (-3,-0.3) node[xscale=1, yscale=1] {$H_{1,1,2}$};
\draw[color=black] (4,5) node[xscale=1, yscale=1] {$H_{1,1,1}$};
\end{tikzpicture}
\end{minipage}
\hfill
\begin{minipage}[c]{.4\linewidth}
\begin{tikzpicture}[scale=0.3]

\fill[line width=2.pt,line space=10pt, pattern=my north east lines] (0,-5) -- (0,8) -- (8,8) -- (8,-5) -- cycle ;
\fill[line width=2.pt,line space=10pt, pattern=my north west lines] (-5,3) -- (8,3) -- (8,-5) -- (-5,-5) -- cycle ;

\draw  [line width=1.pt](0,-5)-- (0,8);
\draw  [line width=1.pt](-5,3) -- (8,3);

\draw  [line width=1.pt,->](-5,0)--(8.3,0) ;
\draw  [line width=1.pt,->](0,-5)--(0,8.3) ;

\draw[color=black] (-1,-2) node[xscale=1, yscale=1] {$H_{2,2,1}$};
\draw[color=black] (5,3) node[xscale=1, yscale=1] {$H_{2,2,2}$};
\draw[color=black] (5,-2) node[xscale=1, yscale=1] {$Q_{2,2}$};
\end{tikzpicture}
\end{minipage}
\caption{\label{figure5} Step 3 ---  The decompositions of $Q_{1,1}$ and $Q_{2,2}$.}
\end{center}
\end{figure}
\qedsymbol


\subsection{Remarks on related control problems}

In this subsection we discuss some related control problems for the heat equation. 
The first one is mentionned in \cite{KNT} and the others in \cite{HKT}.

\subsubsection{Smooth boundary control}
In \cite{KNT}, Kellay, Normand and Tucsnak gave a characterization of the reachable space when the control is smooth. 
Let $s \in \N$ and denote by $ W^{s,2}_{\mathrm{L}}\left((0, \, \tau), \, \C^2\right)$ the Sobolev type space given by 
$$ W^{s,2}_{\mathrm{L}}\left((0, \, \tau), \, \C^2\right):= \enstq{v \in L^2\left((0, \, \tau), \, \C^2\right)}{\substack{\forall 1 \leq k \leq s,\  \frac{d^k v}{dt^k} \in L^2\left((0, \, \tau), \,  \C^2\right) \\
\text{and} \quad \forall 0 \leq k \leq s-1, \ \frac{d^k v}{dt^k}(0)=0}}.$$
Write also 
$$\left(A^2(\Delta) + A^2(\pi-\Delta)\right)^{(s)}:= \enstq{f \in A^2(\Delta) + A^2(\pi-\Delta)}{\forall 1 \leq k \leq s, \ f^{(2k)} \in A^2(\Delta) + A^2(\pi-\Delta)}$$
Then, combining their propositions 5.1 and 7.1, they proved that 
\begin{equation}
\label{reachSobolev}
\mathrm{Ran} \left({\Phi_\tau}_{| W^{s,2}_{\mathrm{L}}\left((0, \, \tau),\,  \C^2\right)}\right) = \left(A^2(\Delta) + A^2(\pi-\Delta)\right)^{(s)}.
\end{equation}
Similarly as above, we let $(A^2)^{(s)}(D)$ be the space of functions $f \in A^2(D)$ such that $f^{(2k)} \in A^2(D)$ for all $1\leq k \leq s$. 
Using Corollary \ref{reach}, result \eqref{reachSobolev} immediately leads to
\begin{cor}
We have 
$\mathrm{Ran} \left({\Phi_\tau}_{| W^{s,2}_{\mathrm{L}}\left((0, \, \tau), \, \C^2\right)}\right)= (A^2)^{(s)}(D)$.
\end{cor}

\subsubsection{Neumann conditions}
It is also possible to ask for a description of the reachable space for other types of boundary conditions. For Neumann boundary conditions, the result follows directly from Corollary \ref{reach} and a trick used in \cite{HKT}. We remind that we are searching the reachable space of the equation 
\begin{equation}
\label{HE-Neumann}
\left\lbrace
	\begin{aligned}
		&\frac{\partial y}{\partial t}(t,x)-		\frac{\partial^2 y}{\partial x^2}(t,x) = 0  \qquad &t >0, \ x\in (0,\pi),& \\
		&\frac{\partial y}{\partial x}(t,0)=u_0(t),\  \  \frac{\partial y}{\partial x}(t,\pi)=u_\pi(t) \qquad & t >0, &\\
		&y(0,x)= g(x) \qquad &x \in (0, \pi),
	\end{aligned}
\right.
\end{equation}
Again, for every initial condition $g \in L^2(0, \, \pi)$ and every control function $u=(u_0, \, u_\pi) \in L^2(\R_+, \C)$, the previous equation \eqref{HE-Neumann} admits a unique solution $y \in C\left([0, \, + \infty), L^2(0, \pi)\right)$. 
We denote by $\Phi^{NN}_\tau$ the controllability map associated to this equation. Let $\mathcal{D}(D)$ be the Dirichlet space on $D$, which consists of all holomorphic functions $F$ in $D$ such that $F' \in A^2(D)$. As noted in \cite[Prop. 5.2]{HKT}, $w$ is a solution of \eqref{HE-Neumann} if and only if $y=\frac{\partial w}{\partial x}$ is a solution of \eqref{HE} with initial condition $f= g'$. Thus the next result follows. 

\begin{cor} 
\label{reachN}
We have $\mathrm{Ran} \Phi^{NN}_\tau = \mathcal{D}(D)$.
\end{cor}

\subsubsection{Dirichlet condition at one end}
Now, we are looking for the reachable space for Dirichlet boundary condition at one end on the interval i.e.\ the reachable space of the equation 
\begin{equation}
\label{HE-Dir-one-end}
\left\lbrace
	\begin{aligned}
		&\frac{\partial y}{\partial t}(t,x)-		\frac{\partial^2 y}{\partial x^2}(t,x) = 0  \qquad &t >0, \ x\in (0,\pi),& \\
		& y(t,0)=0,\  \  y(t,\pi)=u_\pi(t) \qquad & t >0, &\\
		&y(0,x)= g(x) \qquad &x \in (0, \pi).
	\end{aligned}
\right.
\end{equation}
Note that if $y$ is a solution of \eqref{HE-Dir-one-end}, then its odd extension $\widetilde{y}$ to $[-\pi, \, \pi]$ is a solution of 
\begin{equation}
\left\lbrace
	\begin{aligned}
		&\frac{\partial \widetilde{y}}{\partial t}(t,x)- \frac{\partial^2 \widetilde{y}}{\partial x^2}(t,x) = 0  \qquad &t >0, \ x\in (-\pi,\pi),& \\
		& \widetilde{y}(t,-\pi)=-u_\pi(t),\  \  \widetilde{y}(t,\pi)=u_\pi(t) \qquad & t >0, &\\
		&\widetilde{y}(0,x)= g(x) \qquad &x \in (-\pi, \pi).
	\end{aligned}
\right.
\end{equation}
Denote by $D_2$ the square $D_2=\enstq{z=x+iy \in \C}{\abs{x}+\abs{y}< \pi}$. 
So, using Corollary \ref{reach} on $[-\pi, \, \pi]$ we obtain $\widetilde{y} \in A^2(D_2)$ and $\widetilde{y}(-z)=-\widetilde{y}(z), \forall z \in D_2$. Respectively, if $\widetilde{y} \in A^2(D_2)$ and $\forall z \in D_2, \ \widetilde{y}(-z)=-\widetilde{y}(z)$, then $\widetilde{y}_{|[0, \, \pi]}$ is a solution of \eqref{HE-Dir-one-end}. Thus, we have proved 
\begin{cor} 
\label{reach-Dir-one-end}
Denote by $\mathrm{Ran} \Phi^{0D}_\tau$ the reachable space of \eqref{HE-Dir-one-end}. \\
We have $\mathrm{Ran} \Phi^{0D}_\tau = \enstq{f \in A^2(D_2)}{\forall z \in D_2, \ f(-z)=-f(z)}:=A^2_{\mathrm{odd}}(D_2)$.
\end{cor}

\subsubsection{Neumann condition at one end}
Finally, let $\mathrm{Ran} \Phi^{0N}_\tau$ be the reachable space for Neumann boundary condition at one end, i.e the reachable space of \eqref{HE-Neumann} with $u_0=0$. Using an even extension and with the same kind of arguments as in the previous case, we obtain 
\begin{cor} 
\label{reach-N-one-end}
We have 
\begin{align*}
\mathrm{Ran} \Phi^{0N}_\tau &= \enstq{f \in \mathcal{D}(D_2)}{\forall z \in D_2, \ f(-z)=f(z)}:=\mathcal{D}_{\mathrm{even}}(D_2)\\
&=\enstq{f \in \mathrm{Hol}(D_2)}{f' \in A^2_{\mathrm{odd}}(D_2)}.
\end{align*}
\end{cor}

\subsection*{Acknowledgements}
The authors would like to thank Joaquim Ortega-Cerd\`a for some helpful discussions.

\bibliographystyle{alpha} 
\bibliography{biblio} %
\end{document}